\documentclass[11pt]{article}

\usepackage{amssymb, amsmath, amsfonts, amsthm, graphics, amscd}
\usepackage{authblk}
\newcommand{\mathsym}[1]{{}}

\usepackage{tikz}
\usepackage{comment}
\usepackage{algorithm,algorithmicx,algpseudocode}
\usepackage{pdflscape}
\usetikzlibrary{positioning,chains,fit,shapes,calc}

\newtheorem{theorem}{Theorem}[section]

\newtheorem{lemma}[theorem]{Lemma}

\begin{document}

\title{Anomaly Detection and Correction\\in Large Labeled Bipartite Graphs
\\[0.9em]
\small{SUMMER $2016$ INTERN PROJECT}
}
\author[1]{R. W. R. Darling}
\author[2]{Mark L. Velednitsky}
\affil[1]{Mathematics Research Group, National Security Agency, USA}
\affil[2]{University of California, Berkeley, USA}
\date{\today}
\maketitle

\begin{abstract}
Binary classification problems can be naturally modeled as bipartite graphs, where we attempt to classify right nodes based on their left adjacencies. We consider the case of labeled bipartite graphs in which some labels and edges are not trustworthy. Our goal is to reduce noise by identifying and fixing these labels and edges.

We first propose a geometric technique for generating random graph instances with untrustworthy labels and analyze the resulting graph properties. We focus on generating graphs which reflect real-world data, where degree and label frequencies follow power law distributions.

We review several algorithms for the problem of detection and correction, proposing novel extensions and making observations specific to the bipartite case. These algorithms range from math programming algorithms to discrete combinatorial algorithms to Bayesian approximation algorithms to machine learning algorithms.

We compare the performance of all these algorithms using several metrics and, based on our observations, identify the relative strengths and weaknesses of each individual algorithm.
\end{abstract}

\tableofcontents

\section{Problem Statement}

\subsection{Anomaly Detection Literature Background}

The general problem of detecting anomalies in graphs has recieved wide interest. 
Several comprehensive surveys have been published, including Chandola, Banerjee, and Kumar \cite{chand}, as well as Akoglu, Tong and Koutra \cite{ako}.
Following the terminology of the latter, anomaly detection literature can be broken down by the type of graphs studied.
The first distinction is between static and dynamic graphs. In our case, we deal with static graphs: graphs that do not evolve over time.
The second distinction is between attributed and plain graphs. In our case, we are working on attributed graphs: in additional to structural properties of the graph, we care about its labels.
The problem of unifying notions of similarity based on structure and context has received great attention and given rise to algorithms such as SimRank \cite{jeh}.

In many applications, the anomaly detection problem is a natural complement of clustering. 
Anomalies are nodes which are not well described by any clustering.
Algorithms for anomaly detection can be inspired by or derived from clustering algorithms. We note these connections throughout our work.
Anomaly detection on bipartite graphs is closely related to the problem of co-clustering, addressed in \cite{dhi} and \cite{cha}.

The specific case of anomaly detection on bipartite graphs is addressed, for example, by Sun et al \cite{sun}, and by Dai et al \cite{dai}.
Labeled bipartite graphs are recognized as a natural model for mutual reinforcement in a truth discovery context. 
Given a set of evidence and a set of conclusions, some of which conflict, a piece of evidence is considered more reliable if it supports many valid conclusions and a conclusion is considered more valid if it is supported by more reliable evidence.
Examples of this structure appear in online information synthesis \cite{yin} and document summarization \cite{zha}, as well as in data repair for inconsistent databases \cite{afr,are}.
Additional cases which fit our bipartite model will be described in \ref{subsectApplications}.

The present work differs from others in the literature in its emphasis on algorithms for correcting anomalous vertex labels, rather than merely detecting them.

\subsection{Hypergraph Formulation}
Let $\mathcal{R}$ be a collection of subsets of the universe $\mathcal{U}$, also called a \textbf{hypergraph} on the vertex set $\mathcal{U}$. 
For illustrating our construct, we give a sample application (for a long list of potential applications, see \ref{subsectApplications}):

\textbf{DNA Mutations: } In this application, $\mathcal{U}$ is a universe of DNA sequences of individuals;
By examining sets of DNA sequences that share the same mutations, we would like to determine which genetic mutations correspond to which disease predispositions.
Each element of $\mathcal{R}$ is the subset of $\mathcal{U}$ that represents DNA sequences which share a given genetic mutation.

Each element $e \in \mathcal{U}$ has a true color in the set of colors $C$. Call this color $c(e)$, where $c : \mathcal{U} \to C$. This color is unobserved.
In DNA Mutations, a color would represent having a specific disease predisposition.
The color is considered unobserved because many of the individuals may not yet express the disease phenotype or it is difficult to diagnose.
The individual's medical information may not be known or tested, at all.
 
Each set $S \in \mathcal{R}$ has a \textbf{true color} $\bar{c}(S) \in C \cup \{\mho\}$. Here $\mho$ represents a ``wild'' color (explained in \ref{typesOfIrregularities}).
The true color $\bar{c}(S)$ is unobserved. Instead, a coloring $\tilde{c}(S) \in C$ is proposed to us
for each $S \in \mathcal{R}$.
In DNA Mutations, $\bar{c}(S)$ means that a gene mutation causes a specific disease predisposition, and
$\mho$ might stand for ``uncategorized,'' meaning a genetic mutation which is not predictive of disease predisposition.
Lab data gives us a prior belief about which genes are associated with which disease predispositions ($\tilde{c}(S)$).

In a conflict-free setting, each set consists only of elements which match its true color, and its true color 
matches its proposed color. That is, $$e \in S \implies c(e) = \bar{c}(S) = \tilde{c}(S).$$ 
However, in our applications there are irregularities where this rule is violated.
For example, in DNA Mutations there may be an error in testing where a disease predisposition is attributed to a mutation which is actually harmless.
Given the collection of subsets $\mathcal{R}$ and their proposed colors, $\tilde{c}(S)$ for each $S \in \mathcal{R}$, we are attempting to detect the irregularities and, consequentially, the true colors of the subsets: $\bar{c}(S)$ for $S \in \mathcal{R}$.

\subsection{Graph Expression} \label{graph_expression}
The information in the set formulation can be translated to an equivalent graph formulation. We construct a bipartite graph $G$. With each element $e \in \mathcal{U}$ we associate a left node $\ell_e$. With each set $S \in \mathcal{R}$ we associate a right node $r_S$. There is an edge from $\ell_e$ to $r_S$ if $e \in S$. We can initialize the graph with a coloring $\tilde{c}(S)$ on the right nodes.

In some applications, it will be useful to think of the graph slightly differently. 
Let $G'$ be a graph constructed by adding a new node $r_S'$ for each right node $r_S$ in $G$ and connecting $r_S'$ to $r_S$. This creates $|R|$ new nodes and edges. We call $G'$ the Auxiliary graph. We can think of the $r_S'$ as observed nodes and the rest of the nodes as hidden. While we may change the colors of the $r_S$, the colors of the $r_S'$ stay fixed.

\subsection{Applications} \label{subsectApplications}
This model has wide applications to contexts where classification is supported by shared evidence. We have already described the DNA Mutations application. In the following table, we briefly list a few others:

\begin{tabular}{|l|ccc|}
\hline
Application & Left Node & Right Node & Colors \\
\hline
Academic & Author & Paper & Discipline \\
Advertisement & Target & Campaign & Product Type \\
Commerce & Customer & Business & Location \\
Commerce & Product & Store & Category \\
Entertainment & Fan & Band & Genre \\
Entertainment & Watcher & Show & Genre \\
Finance & Trader & Security & Category \\
Legal & Case & Lawyer & Specialty Area \\
Medical & DNA Sequence & Mutation & Disease Predisposition \\
Medical & Patient & Doctor & Practice/Network \\
Medical & Symptom & Patient & Rare Illness \\
Politics & Individual & Endorser & Affiliation \\
Programming & Subroutine & Code & Application \\
Security & Crime & Criminal & Threat Level \\
Topic Modeling & Words & Sentences & Topic \\
\hline
\end{tabular}

We see a common structures emerging. Most right nodes can be assigned a single color. When a left node is associated with several right nodes, we can reasonably expect most of those right nodes to have the same color. Some right nodes may not fit well into any class and some left nodes may connect to right nodes across classes, but we expect the number of exceptions not to be too large (less than $20\%$ of the graph, for example). Finally, in realistic data we may expect one set of nodes to have degrees following a power law distribution.

\subsection{Types of Irregularities} \label{typesOfIrregularities}

There are three kinds of irregularities:

\begin{itemize}
	\item 
\textbf{Wildness: } This occurs when a right node (set) is a neighbor of random left nodes (elements). In our set theory notation, $\bar{c}(S) = \mho$.

	\item 
\textbf{Mislabelings}:  This occurs when the right node (set) does give consistent information but has been labeled with the wrong color. 
In our set theory notation,
the proposed color $\tilde{c}(S)$  differs from the true color $\bar{c}(S) \neq \mho$. 

	\item 
\textbf{Misattribution: } 
This occurs when a left node (element) has been identified with a right node (set) when it should not have been. In our set theory notation, 
$\bar{c}(S) \neq \mho$, but $c(e) \neq \bar{c}(S)$ for some $e \in S$.

\end{itemize}

In this paper, we will consider all three types of irregularities and evalute the effectiveness of our algorithms in detecting individual irregularities and combinations thereof.

\section{Graph Images}

There are four images:
\begin{enumerate}
	\item A clean graph with no anomalies (Figure \ref{image_clean}).
	\item A graph with anomalies added (Figure \ref{image_dirty}). For example, $y_1$ is mislabeled, $y_{11}$ is wild, and several misattributed edges have been added.
	\item What the input graph looks like (Figure \ref{image_input}). Notice that wilds are assigned a color label and left nodes are unlabeled.
	\item The Auxiliary graph (Figure \ref{aux_input}): a convenient construction.
\end{enumerate}

\definecolor{lightgreen}{RGB}{80,160,80}
\definecolor{lightblue}{RGB}{80,80,160}
\definecolor{lightred}{RGB}{160,80,80}

\begin{figure} 
\centering
\begin{tikzpicture}[thick,
  every node/.style={draw,circle},
  bnode/.style={fill=blue},
  gnode/.style={fill=green},
  rnode/.style={fill=red},
  blnode/.style={fill=lightblue},
  glnode/.style={fill=lightgreen},
  rlnode/.style={fill=lightred},  
]
\begin{scope}[start chain=going below,node distance=7mm]
\foreach \i in {1,2,...,4}
  \node[glnode,on chain] (l\i) [label=left:$x_{\i}$] {};
\foreach \i in {5,6,...,7}
  \node[blnode,on chain] (l\i) [label=left:$x_{\i}$] {};
\foreach \i in {8,9,...,11}
  \node[rlnode,on chain] (l\i) [label=left:$x_{\i}$] {};
\end{scope}
\begin{scope}[xshift=3cm,start chain=going below,node distance=7mm]
\foreach \i in {1,2,...,4}
  \node[gnode,on chain] (r\i) [label=right:$y_{\i}$] {};
\foreach \i in {5,6,...,7}
  \node[bnode,on chain] (r\i) [label=right:$y_{\i}$] {};
\foreach \i in {8,9,...,10}
  \node[rnode,on chain] (r\i) [label=right:$y_{\i}$] {};
\end{scope}
\draw (r1) -- (l1);
\draw (r1) -- (l2);
\draw (r1) -- (l4);
\draw (r2) -- (l2);
\draw (r2) -- (l3);
\draw (r3) -- (l3);
\draw (r4) -- (l3);
\draw (r4) -- (l4);
\draw (r5) -- (l5);
\draw (r6) -- (l7);
\draw (r5) -- (l6);
\draw (r5) -- (l6);
\draw (r7) -- (l5);
\draw (r6) -- (l5);
\draw (r8) -- (l8);
\draw (r9) -- (l9);
\draw (r9) -- (l10);
\draw (r10) -- (l11);
\draw (r10) -- (l8);
\draw (r8) -- (l10);
\draw (r8) -- (l8);
\draw (r8) -- (l11);
\end{tikzpicture}
\caption{The Bipartite Graph with No Noise}\label{image_clean}
\end{figure}
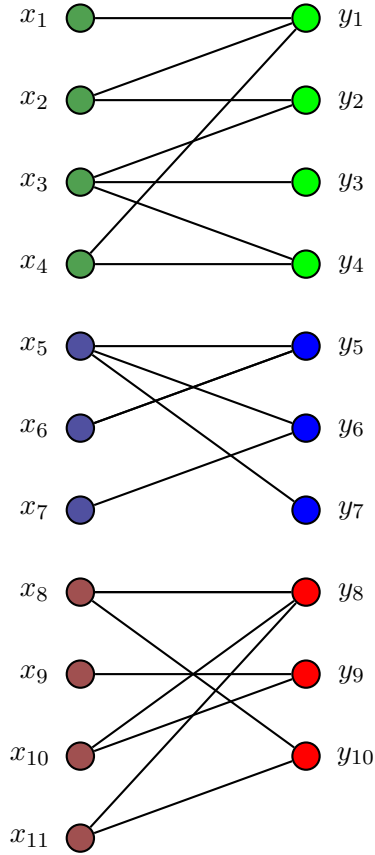

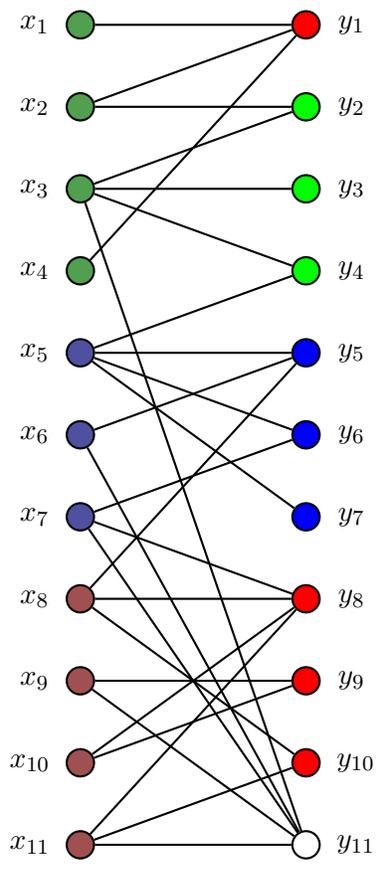
\begin{figure} 
\centering
\begin{tikzpicture}[thick,
  every node/.style={draw,circle},
  bnode/.style={fill=blue},
  gnode/.style={fill=green},
  rnode/.style={fill=red},
  blnode/.style={fill=lightblue},
  glnode/.style={fill=lightgreen},
  rlnode/.style={fill=lightred},
  wnode/.style={fill=white},
]
\begin{scope}[start chain=going below,node distance=7mm]
\foreach \i in {1,2,3,4}
  \node[glnode,on chain] (l\i) [label=left:$x_{\i}$] {};
\foreach \i in {5,6,7}
  \node[blnode,on chain] (l\i) [label=left:$x_{\i}$] {};
\foreach \i in {8,9,10,11}
  \node[rlnode,on chain] (l\i) [label=left:$x_{\i}$] {};
\end{scope}
\begin{scope}[xshift=3cm,start chain=going below,node distance=7mm]
\node[rnode,on chain] (r1) [label=right:$y_1$] {};
\foreach \i in {2,3,4}
  \node[gnode,on chain] (r\i) [label=right:$y_{\i}$] {};
\foreach \i in {5,6,7}
  \node[bnode,on chain] (r\i) [label=right:$y_{\i}$] {};
\foreach \i in {8,9,10}
  \node[rnode,on chain] (r\i) [label=right:$y_{\i}$] {};
\node[wnode,on chain] (r11) [label=right:$y_{11}$] {};
\end{scope}
\draw (r1) -- (l1);
\draw (r1) -- (l2);
\draw (r1) -- (l4);
\draw (r2) -- (l2);
\draw (r2) -- (l3);
\draw (r3) -- (l3);
\draw (r4) -- (l3);
\draw (r4) -- (l5);
\draw (r5) -- (l5);
\draw (r6) -- (l7);
\draw (r5) -- (l6);
\draw (r5) -- (l8);
\draw (r7) -- (l5);
\draw (r6) -- (l5);
\draw (r8) -- (l7);
\draw (r9) -- (l9);
\draw (r9) -- (l10);
\draw (r10) -- (l11);
\draw (r10) -- (l8);
\draw (r8) -- (l10);
\draw (r8) -- (l8);
\draw (r8) -- (l11);
\draw (r11) -- (l3);
\draw (r11) -- (l6);
\draw (r11) -- (l7);
\draw (r11) -- (l9);
\draw (r11) -- (l11);
\end{tikzpicture}
\caption{The Bipartite Graph with Noise}\label{image_dirty}
\end{figure}

\begin{figure} 
\centering
\begin{tikzpicture}[thick,
  every node/.style={draw,circle},
  bnode/.style={fill=blue},
  gnode/.style={fill=green},
  rnode/.style={fill=red},
  blnode/.style={fill=lightblue},
  glnode/.style={fill=lightgreen},
  rlnode/.style={fill=lightred},
  wnode/.style={fill=white},
]
\begin{scope}[start chain=going below,node distance=7mm]
\foreach \i in {1,2,3,4}
  \node[wnode,on chain] (l\i) [label=left:$x_{\i}$] {};
\foreach \i in {5,6,7}
  \node[wnode,on chain] (l\i) [label=left:$x_{\i}$] {};
\foreach \i in {8,9,10,11}
  \node[wnode,on chain] (l\i) [label=left:$x_{\i}$] {};
\end{scope}
\begin{scope}[xshift=3cm,start chain=going below,node distance=7mm]
\node[rnode,on chain] (r1) [label=right:$y_1$] {};
\foreach \i in {2,3,4}
  \node[gnode,on chain] (r\i) [label=right:$y_{\i}$] {};
\foreach \i in {5,6,7}
  \node[bnode,on chain] (r\i) [label=right:$y_{\i}$] {};
\foreach \i in {8,9,10}
  \node[rnode,on chain] (r\i) [label=right:$y_{\i}$] {};
\node[gnode,on chain] (r11) [label=right:$y_{11}$] {};
\end{scope}
\draw (r1) -- (l1);
\draw (r1) -- (l2);
\draw (r1) -- (l4);
\draw (r2) -- (l2);
\draw (r2) -- (l3);
\draw (r3) -- (l3);
\draw (r4) -- (l3);
\draw (r4) -- (l5);
\draw (r5) -- (l5);
\draw (r6) -- (l7);
\draw (r5) -- (l6);
\draw (r5) -- (l8);
\draw (r7) -- (l5);
\draw (r6) -- (l5);
\draw (r8) -- (l7);
\draw (r9) -- (l9);
\draw (r9) -- (l10);
\draw (r10) -- (l11);
\draw (r10) -- (l8);
\draw (r8) -- (l10);
\draw (r8) -- (l8);
\draw (r8) -- (l11);
\draw (r11) -- (l3);
\draw (r11) -- (l6);
\draw (r11) -- (l7);
\draw (r11) -- (l9);
\draw (r11) -- (l11);
\end{tikzpicture}
\caption{The Bipartite Graph as Input}\label{image_input}
\end{figure}
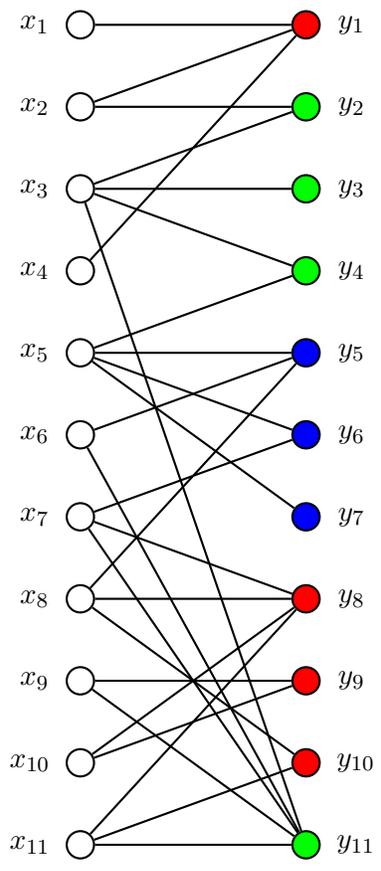

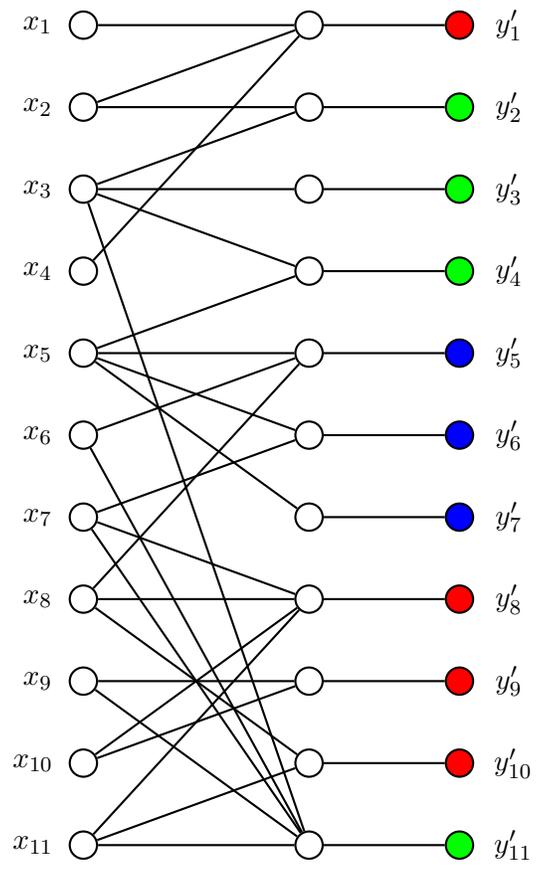
\begin{figure} 
\centering
\begin{tikzpicture}[thick,
  every node/.style={draw,circle},
  bnode/.style={fill=blue},
  gnode/.style={fill=green},
  rnode/.style={fill=red},
  blnode/.style={fill=lightblue},
  glnode/.style={fill=lightgreen},
  rlnode/.style={fill=lightred},
  wnode/.style={fill=white},
]
\begin{scope}[start chain=going below,node distance=7mm]
\foreach \i in {1,2,3,4}
  \node[wnode,on chain] (l\i) [label=left:$x_{\i}$] {};
\foreach \i in {5,6,7}
  \node[wnode,on chain] (l\i) [label=left:$x_{\i}$] {};
\foreach \i in {8,9,10,11}
  \node[wnode,on chain] (l\i) [label=left:$x_{\i}$] {};
\end{scope}
\begin{scope}[xshift=3cm,start chain=going below,node distance=7mm]
\node[wnode,on chain] (r1) {};
\foreach \i in {2,3,4}
  \node[wnode,on chain] (r\i) {};
\foreach \i in {5,6,7}
  \node[wnode,on chain] (r\i) {};
\foreach \i in {8,9,10}
  \node[wnode,on chain] (r\i) {};
\node[wnode,on chain] (r11) {};
\end{scope}
\begin{scope}[xshift=5cm,start chain=going below,node distance=7mm]
\node[rnode,on chain] (rr1) [label=right:$y'_1$] {};
\foreach \i in {2,3,4}
  \node[gnode,on chain] (rr\i) [label=right:$y'_{\i}$] {};
\foreach \i in {5,6,7}
  \node[bnode,on chain] (rr\i) [label=right:$y'_{\i}$] {};
\foreach \i in {8,9,10}
  \node[rnode,on chain] (rr\i) [label=right:$y'_{\i}$] {};
\node[gnode,on chain] (rr11) [label=right:$y'_{11}$] {};
\end{scope}

\draw (r1) -- (l1);
\draw (r1) -- (l2);
\draw (r1) -- (l4);
\draw (r2) -- (l2);
\draw (r2) -- (l3);
\draw (r3) -- (l3);
\draw (r4) -- (l3);
\draw (r4) -- (l5);
\draw (r5) -- (l5);
\draw (r6) -- (l7);
\draw (r5) -- (l6);
\draw (r5) -- (l8);
\draw (r7) -- (l5);
\draw (r6) -- (l5);
\draw (r8) -- (l7);
\draw (r9) -- (l9);
\draw (r9) -- (l10);
\draw (r10) -- (l11);
\draw (r10) -- (l8);
\draw (r8) -- (l10);
\draw (r8) -- (l8);
\draw (r8) -- (l11);
\draw (r11) -- (l3);
\draw (r11) -- (l6);
\draw (r11) -- (l7);
\draw (r11) -- (l9);
\draw (r11) -- (l11);
\draw (r1) -- (rr1);
\draw (r2) -- (rr2);
\draw (r3) -- (rr3);
\draw (r4) -- (rr4);
\draw (r5) -- (rr5);
\draw (r6) -- (rr6);
\draw (r7) -- (rr7);
\draw (r8) -- (rr8);
\draw (r9) -- (rr9);
\draw (r10) -- (rr10);
\draw (r11) -- (rr11);
\end{tikzpicture}
\caption{The Auxiliary Bipartite Graph}\label{aux_input}
\end{figure}

\section{Generative Models}

\subsection{Generating Random Instances}
To generate our random graph instances, we would like to admit parameters such that the three types of irregularities appear in different predetermined proportions. We assign variables to these three parameters:
\begin{itemize}
	\item Let $\omega$ be the target fraction of \textbf{wild nodes}.
	\item Let $\lambda$ be the target fraction of \textbf{mislabeled nodes}.
	\item Let $\alpha$ be the target fraction of \textbf{misattributed edges}.
\end{itemize}
We will introduce three models: the sequential model, the circle model, and the circle model with power law degree distributions. The latter two models attempt to emulate the ``geometric'' nature of real graphs: nodes are more likely to form new connections with the neighbors of their neighbors than with distant nodes in the graph. The third model attempts to emulate the tendency of real graphs to have a few nodes of very high degree and many nodes of small degree. The structure of real-world graphs is studied extensively in \cite{coh}. In \cite{fel}, circle models are used for graph visualization. It is natural to extend the idea to graph generation.

\subsubsection{The Sequential Model}
As the name suggests, in the sequential model we start with a conflict-free graph (with no irregularities) and introduce the irregularities sequentially. We start with this model since it gives us an intuitive sense of which order to introduce the irregularities. This intuition will carry over to later models.

For the sequential model, we proceed as follows:
\begin{enumerate}
	\item Generate $d$ disconnected, random bipartite graphs on vertex sets $L_i \cup R_i$, $1 \leq i \leq d$, where $d$ is the number of colors. There are many known algorithms for this task. Label the right nodes $R_i$ in bipartite graph $i$ with color $c_i$, where all the $\{c_i\}$
are distinct. Declare the disjoint union $L := \bigcup_{1 \leq i \leq d} L_i$ to be the set of left nodes.
	\item Introduce a set $R_0$ of wild nodes such that the total fraction of wild nodes is $\omega$. Connect these nodes to random left nodes in $L$.
The set of right nodes is $R := \bigcup_{0 \leq i \leq d} R_i$.
	\item For each edge in the graph, with probability $\alpha$ delete the edge and replace it with a new edge incident to the same right node and a random left node.
	\item For each right node in the graph, with probability $\lambda$ randomly assign the label of the node to a new color.
\end{enumerate}

The key benefit of this model is that it allows us to control precisely each parameter. The downside is that it ignores the geometric nature of the problem and the substructure we would expect on the bipartite subgraphs. We expect these instances to be relatively easy to solve since their irregularities are random and not structured as they will be in the circle model.

\subsubsection{The Circle Model}
The circle model attempts to emulate geometric graph constructions by placing the nodes on a circle, and determining connectivity as a function of physical proximity.

To create a graph with $C$ colors, first create a circle $Q$ of circumference $C$. Place two sets of points along $Q$ according to a predetermined probability distribution (typically uniform) to represent right and left nodes. These two distributions can be independent or dependent, as desired. The circle $Q$ can be naturally divided into $C$ segments and the location of each node determines its true color. We go through the right nodes and, with probability $\lambda$, give them a random color label instead of their true color label.

For each fixed right node $r \in R$, select randomly from the set of left nodes $L$ a neighbourhood $\Gamma_L(r)$ for $r$ as follows. With probability $\omega$ our selection is uniformly random. With probability $1-\omega$ it is biased towards nodes that are nearby; to achieve this, a node at a distance $\delta$ is selected with probability proportional to some non-increasing function of distance $f(\delta)$. Reasonable choices for $f$ include an exponential decay function, a step function, or a threshold function ($f(\delta) = 1_{\{ \delta < \bar{\delta} \}} $). The upside of using a threshold is that it can speed up computation by limiting the number of left nodes examined.

Note that we did not explicitly incorporate $\alpha$ into our model. The number of misattributed edges will depend on the slope of $f$. The only way to express it precisely would be with a double integral over the circle, which may be difficult to calculate efficiently. If we want $\alpha$ small, then we should choose an $f$ which decreases quickly so that right nodes will be biased towards nearby left nodes. On the other hand, if we want $\alpha$ large, then we should choose an $f$ which decreases slowly so that faraway left nodes sill have a chance of being included in the neighborhood of right nodes.

The key benefit of this model is that it constructs difficult instances where the irregularities, especially the misattributions, arise from an underlying physical process. This means that a right node with misattributions is likely to have several misattributions of the same type. It is also likely to be near other right nodes with similar misattributions. Wild nodes can be thought of as nodes in the center of the circle, equidistant from all the left nodes. The downsides of this model are that the construction is slow (since we examine all pairs of right and left nodes) and it does not allow us to control $\alpha$ explicitly.

\subsubsection{Circle Model with Power Law Distribution for Colors and Right Vertex Degrees}

In some applications of interest, the color distribution of right nodes may have heavy tails, such as those given by Zipf's Law. Furthermore, while the degree distribution of left nodes is approximately Poisson, the degree distribution of right nodes may also have heavy tails. We will modify the Circle Model to better reflect these power law distributions.

Our first modification is to distribute the colors according to a preferential attachment model. We place the nodes sequentially. When deciding on the color of a node $k$, we consider the existing color distribution. The probability of giving node $k$ the color $c_j$ is directly proportional to
\[
|\{r_i \in R|i<k,c(r_i)=c_j\}| + \chi,
\]
where $\chi$ is a small constant. For example, consider the case where $\chi = 0.25$ and there are $40$ colors. The first node is given a color uniformly at random. The second node now has probability $\frac{1 + 0.25}{1 + 40 \cdot 0.25}$ of getting assigned the same color as the first. Once we have decided on the color of a node, we give it a location uniformly at random on the circle within the region of that color. Note that in this preferential attachment model, the order in which nodes are selected does not affect the final distribution. For convenience, we place all the right nodes, then the left nodes. It can be shown that the probability a color is not represented by any node tends to $0$ as the number of nodes tends to infinity.

Next, we decide on the neighbors of nodes similar to the Circle Model. In this case, we proceed in two steps. First, for each left node we select a small constant number of neighbors (ex. one or two). Next, for each right node we sample from a heavy tail distribution $Z$, such that $p(Z > z) = \frac{1}{z}$, and give the right node that many neighbors. If the right node is wild and was assigned an edge in the first step, we randomly reassign that edge. Note that both steps are necessary. If we do not first select two neighbors of each left node, then many left nodes will end up without neighbors. If we do not reassign edges that linked to wilds in the first step, then our wild nodes will end up being bias towards linking to the same color, making them not really wild. The random assignments, themselves, takes place according to the same $f$ as in the basic Circle Model.

\subsection{Measuring the Difficulty of the Problem}
\subsubsection{Difficulty Metrics}
In a graph where we are given the true labels, we can easily calculate $\omega$, $\lambda$, and $\alpha$. We would like to develop some metrics to measure the difficulty of test instances where these parameters are unknown. A few graph metrics of interest:

\begin{enumerate}
	\item The \textbf{color distribution} of the right nodes. Generally, if the number of nodes per color follows
a multinomial$(|R|, (1/d, 1/d, \ldots, 1/d))$ distribution,
 the problem will be easier to resolve than when the graph has a very asymmetric color distribution.
	\item The \textbf{degree distribution} of the left and right nodes. This includes the average degree and the maximum degree. Generally, the problem becomes more difficult as the average degree of nodes decreases and as the distribution of degrees becomes less uniform. In realistic applications, the distribution of right node degrees may follow a power law.
	\item The \textbf{component distribution} of the whole bipartite graph. Without additional a-priori assumptions, it is not possible to infer information between connected components. In realistic applications, the graphs tend to consist of a large connected component with a majority of the nodes and a few small, disconnected components.
	\item Average \textbf{color degree} of left nodes. That is, the average number of colors observed by a left node in its neighbors. Formally, the number of colors adjacent to $\ell$ is the cardinality of the set $\{\tilde{c}(r), r \in \Gamma_R(\ell) \} \subset C$.
The larger this value, the more disagreement exists in the initial graph coloring.
	\item The number of $2$-step paths starting and ending at the same color. This gives us a sense of the \textbf{SNR} (Signal-to-Noise Ratio), since bichromatic paths are the ``noise'' we are trying to eliminate. In graph theory terms, this gives a sense of the amount of connectedness between
what should be disconnected color components.

\end{enumerate}

In a graph where the values of $\omega$, $\lambda$, and $\alpha$ are known, we can estimate the SNR as: $$((1-\lambda)(1-\alpha)(1-\omega))^2.$$ Note that this is a lower bound since it esimates the percentage of the time that a $2$-step path will step between two tame, correctly labeled nodes along properly attributed edges. In reality, the number of monochromatic $2$-step paths will be larger since there are cases where it could include paths between wild nodes or mislabeled nodes along misattributed edges.

\subsubsection{Sample Instances}
The following table gives a rough sense of what our metrics look like on graphs of different sizes generated by different methods.
It gives the average and maximum of the left and right vertex degree, and the SNR as defined above.

\begin{footnotesize}
\begin{center}
\begin{tabular}{|l|c|c|c|c|c|c|c|c|}
\hline
SetID & Edges & LAvg & LMax & RAvg & RMax & ColorDeg & Comps & SNR \\
\hline
Circle Small & 19449 & 3.81 & 15 & 11.44 & 25 & 2.17 & 41 &  0.56 \\
Circle Large & 96793 & 3.8 & 14 & 11.39 & 27 & 1.23 & 219 &  0.89 \\
Power Small & 8590 & 2.86 & 10 & 8.59 & 128 & 1.75 & 70 & 0.59 \\
Power Large & 11093 & 3.70 & 15 & 11.09 & 9116 & 2.35 & 318 & 0.45 \\
\hline
\end{tabular}
\end{center}
\end{footnotesize}

We make a few observations. First, the degrees of the left nodes are relatively small. 
On the right, the power law distribution gives us large maximum degrees without varying the average sigificantly.
More importantly, making the circle graph larger improves SNR. In the power law graph, the SNR decreases as the graph grows.
A graph with $\text{SNR} \approx 0.9$ would be considered an easy problem instance, while an $\text{SNR} \approx 0.4$ would be considered a difficult problem instance.
A graph with $\text{SNR} \approx 0.6$ would be considered a medium difficulty instance.

\section{Algorithms Overview}
Before exploring our algorithms in detail, we give a chart summarizing their properties. The properties are:
\begin{itemize}
	\item \textbf{Global Objective:} Does the algorithm optimize a well-defined global objective function (usually by local computation)?
	\item \textbf{Message Passing:} Can the algorithm be implemented in a message-passing context? A message-passing algorithm is one in which each node only knows its state and its neighbors. At each iteration, it sends a message based on its state and updates its state based on the messages it recieves.
	\item \textbf{Message Size:} Message-passing algorithms could either send a message which is a single color or a vector of probabilities for each color.
	\item \textbf{Termination Optimal:} Does the final result necessarily optimize the global objective?
	\item \textbf{Output Class:} The final output of the algorithm could be either a single color judgement for each node or a vector of probabilities for each color.
	\item \textbf{Randomized Results:} If the algorithm is run again on the same data, could it give a different result?
\end{itemize}

\begin{footnotesize}
\begin{tabular}{|l|cccccc|}
\hline
Algorithm & Global & MPass & Msg & TermOpt & Output & Random\\
\hline
Voting 			& Yes & Yes & Color & No & Color & No \\
Gradient Descent 	& Yes & Yes & Vector & Yes & Vector & No \\
Minimum Cut 		& Yes &  No & N/A & Yes & Color & No \\
Naive Bayesian		&  No & Yes & Vector & N/A & Vector & No \\
Harmonic Function		& Yes & Yes & Vector & Yes & Vector & No \\
Multinomial Bayes	&  No & Yes & Color & N/A & Vector & Yes \\
Neural Network		&  No &  No & N/A & N/A & Color & Yes \\
\hline
\end{tabular}
\end{footnotesize}

\section{Semi-Definite Programming Formulation}

\subsection{Background}
Semi-Definite Programming is a special case of Convex Programming. For a comprehensive review, see Vandenberghe and Boyd \cite{vand}. We choose to investigate Semi-Definite Programming for two reasons. First, because the dot product is a reasonable measure of agreement between two probability distributions. Semi-definite programming is well suited to dealing with dot products in objective functions and constraints. Second, we are inspired by the application of semi-definite programming to approximating the maximum cut problem, where agreement between nodes is naturally modeled with a vector dot product \cite{mgeo}.

\subsection{Objective Function}
If we have $d$ colors, let $r^k = (r_1^k,\ldots,r_d^k)$, with $\sum_i{r_i^k} = 1$, be a $d$-dimentional vector of probabilities, where $r_i^k$ is the probability that right node $v_k$ has color $i$. We define $\ell^j$ similarly for left node $u_j \in L$. We encode the initial coloring of the right node $v_k$ in the vector $h^k$, which is the standard basis vector of $\mathbf{R}^d$ corresponding to the initial color.

Our SDP programming formulation objective will have two terms: the separation and deviation terms. The separation term rewards agreement between adjacent right nodes and left nodes. The deviation term rewards agreement between right nodes and their initial colorings. We are ultimately attempting to strike a balance between honoring our prior beliefs and maximizing agreement. We write the objective function:
\begin{equation}\label{eq:sdpvoting}
\max \left( \sum_{\{u_j, v_k\} \in E}{\ell^j \cdot r^k} + \sum_{v_k \in R}{\tau_k h^k \cdot r^k} \right)
\end{equation}
where $E$ refers to the edge set of the bipartite graph, subject to
\[
\ell^j \cdot \mathbf{1} = 1 \forall j; \quad  r^k \cdot \mathbf{1} = 1 \forall k.
\]

The term on the right is the separation term and the term on the left is the deviation term. The meta-parameters $\tau_k$ represent the strength of the prior at a node and are used to balance the relative importance of the two terms. The larger the value of $\tau_k$, the more importance is given to our prior on a node. If all the $\tau_k$ are too small, then it becomes optimal to set all the $\ell^j$ and $r^k$ equal. If the $\tau_k$ are too large then we end up assigning every right node to its prior.

\subsection{Algorithm}
It turns out that the objective above can be interatively increased by a simple hill-climbing algorithm based on majority votes. The algorithm terminates in a finite number of iterations. We describe the algorithm and give a proof of termination. A key lemma:

\begin{lemma}
If we fix the values of the left or right nodes ($\ell^j$ or $r^k$), then the problem becomes separable in the remaining terms. The remaining variables have an integer optimal solution.
\end{lemma}

\begin{proof}
We inspect the objective function (\ref{eq:sdpvoting}). 
If we fix the right nodes, we can rewrite the objective function (at least, the part not already fixed) as
a linear map
\[
\sum_{u_j \in L} \ell^j \cdot \left(\sum_{v_k \sim u_j}{ r^k}  \right)
\]
If we fix the left nodes, we can rewrite the objective function as a linear map
\[
\sum_{v_k \in R}  r^k \cdot \left(\tau_k  h^k + \sum_{u_j \sim v_k}{ \ell^j} \right)
\]
In both cases, we can separate the outer sum and optimize over the individual $\ell^j$ or $r^j$,
subject to the linear constraints.
Furthermore, if the optimal solution for some $\ell^j$ or $r^j$ has fractional entries, then by linearity there is necessarily a solution of the form $\ell^j = e_i$ or $r^j = e_i$ for some $i$.
\end{proof}

Note that we did not require the $h^k$ to be integral for our proof to work. The lemma will hold even if we introduce fractional priors.

Formally, the voting algorithm is as follows: for each $u_j \in L$, consider $\Gamma_R(u_j)$. Set $\ell^j$ equal to $e_i$ where $i$ maximizes 
\[
\sum_{u_j \sim v_k}{e_i \cdot r^k}.
\]
For all $v_k \in R$, consider $\Gamma_L(v_k)$. Set $r^k$ equal to $e_i$ where $i$ maximizes 
\[
e_i \cdot \left(\tau_k  h^k + \sum_{u_j \sim v_k}{ \ell^j} \right).
\]
Continue this procedure until the objective function is no longer improving.

Intuitively, at each step we take the ``majority vote'' of the colors of the neighbors of a node. If we are at a right node, then the prior also gets $\tau_k$ votes.

As a direct consequence of the previous lemma, the objective improves at each step in the voting algorithm. Furthermore, the objective is always an integer. Since the objective has a natural upper bound, this is sufficent to prove termination (though not necessarily optimality).

\subsection{Parameters/Wildness}
We have not yet determined any rigorous criterion for selecting $\tau_k$. Recall that at each step the node $v_k \in R$ sees a number of votes equal to its degree. It is reasonable to consider values for $\tau_k$ which depend on that degree. In particular, we suggest setting $\tau_k$ to $0.25$ times the degree of $r^k$.

We suggest an ad hoc approach for determining if a node is wild in the voting algorithm. If at any stage the plurality vote accounts for less than half of all votes (i.e. a majority), then we set the node to wild instead of to the plurality color. Once a node is marked as wild, it does not get to vote on its neighbors. The threshold of one half is arbitrary.

\subsection{Summary}
While the Voting Algorithm appears relatively naive, its presence here illustrates that it iteratively increases a fairly natural, rigorously-defined objective arising from the problem. And it does so by passing small messages. It should not be discounted for its simplicity.

\section{Quadratic Model and Gradient Descent Formulation}

\subsection{Background}
For large optimization problems in which finding an exact optimal may be difficult, gradient descent is a frequent method of choice \cite{bott}. 
In our case, we would like to state and optimize a simple linear best-fit objective with constrained variables.
Our objective is convex. We will show that, in our bipartite setting, gradient descent can be easily paralellized and implemented in a message-passing algorithm.
This addresses the well-known challenge of parallelization in gradient descent (and stochastic gradient descent) \cite{niu}.

\subsection{Objective Function}
First consider the binary version of the problem in which there are two colors, black and white.
. The probability a left vertex $u_i$ is colored black is $x_i$,
where $0 \leq x_i \leq 1$. The prior on a right vertex $v_j$ is $y_j \in \{0,1\}$, where $y_j = 1$ means that right vertex $v_j$ was initially assigned color black. We will attempt to minimize the total of a loss function that we will calculate at each right node. The loss function on each right node $v_j$ will 
assign a quadratic penalty to deviation from its prior. Formally, the loss at $v_j$ will be 
\begin{equation} \label{e:rightloss}
\left(\frac{\sum_{u_i \sim v_j}{x_i}}{d_j} - y_j \right)^2,
\end{equation}
where $d_j$ is the degree of node $v_j$. Alternatively the effect of the priors could be made weaker by setting $y_j \in \{0, \rho\}$ for
$\rho<1$.
This is a constrained least squares minimization of the form: choose $X:= (x_i) \in [0,1]^{|R|}$ to minimize
\begin{equation} \label{e:lsm}
(U X - Y)^T (U X - Y)
\end{equation}
where $Y:= (y_j)$, and $U:=(u_{j, i})$, with $u_{j, i}:= 1_{\{u_i \sim v_j\}} d_j^{-1}$; in other words, $U$ is the adjacency matrix, 
rescaled by dividing by right vertex degrees. 

Without the linear constraint there would be an explicit minimizer
$
X = (U^T U)^{-1} U^T Y.
$
The linear constraint that each $x_i \in [0,1]$
makes this a convex quadratic programming problem.
Exact methods have complexity at least $O(|R|\cdot |E|)$, where $|E| = \sum_{i, j} 1_{\{u_i \sim v_j\}}$ is the edge count,
because this is the cost of the unconstrained problem. This is too expensive for big graph applications.
 We present an approximate $O(|E|)$ message passing algorithm.

\subsection{Algorithm}

\subsubsection{Message Passing}
Differentiate (\ref{e:lsm}) with respect to $X$ to obtain the matrix of partial derivatives:
\[
2 (U X - Y)^T U
\]
Split this derivative into a sum. The first summand, the map $X^T \mapsto 2 X^T (U^T U)$, has an effect at left vertex $u_j$ 
which is the sum over all two-step paths from left vertex $u_i$ to $u_j$  of the value $x_i$, 
divided by the square of the degree of
the intermediate right node. The first summand has the effect of propagating the current labelling of left vertices which share a neighbor
with $u_j$. On the other hand the second summand $-2 Y^T U$ is pulling the value $x_j$ at $u_j$
back towards the prior at each of the neighboring right vertices, with a strength inversely proportional to the 
degree of that right neighbor.

\textbf{Message Passing Interpretation: }
The gradient at left node $u_i$ can be computed by summing, across its right neighbors $v_j$, the values 
\begin{equation} \label{e:mplsm}
\frac{2}{d_j} \left( \frac{\sum_{u_{i'} \sim v_j }{x_{i'}}}{d_j} - y_j \right).
\end{equation}
We can calculate this value at each $v_j$ and then pass along the message to the adjacent left nodes. 
If the increment to $x_i$ places the new value outside $[0,1]$, 
in the black and white version we round to the nearer of $0$ and $1$.
After the left nodes update their value, they can pass their value back along to the right nodes in the next iteration.

\subsubsection{Convergence}
Like every quadratic programming problem with linear inequality constraints, this is a convex optimization problem.
Every iteration of the message passing algorithm (\ref{e:mplsm}) pushes the objective function (\ref{e:lsm})
towards the unique global optimum.

\subsubsection{Renormalized Multicolor Version}
The rule (\ref{e:mplsm}) constructs a message for a single color.
Assume there are at least 3 colors. Separate messages may be constructed for each color.
After a left node $u_i$ receives a message for each color $k$, a value of $x_i^k$ is updated. 
In the \textbf{renormalized} version of the algorithm, set negative values to
zero, and then divide each $x_i^k$ by $\sum_k x_i^k$ to obtain a probability distribution across colors at node $u_i$.

\subsection{Parameters/Wildness}
In the multicolor version without renormalization, our $x_i$ are determined separately for each color. 
In the renomalized multicolor version, the output is a probability distribution across colors for each node. 
In either case, we have two options. We could examine the total loss (\ref{e:rightloss}) at each right node. 
Nodes with larger loss are more likely to be wild or mislabeled. 
Alternatively, we could set a threshold $\tau$ and color nodes if their largest $x_i$ is over $\tau$, leaving the rest wild.

The benefit of this gradient algorithm is that its plain version, without $\rho_j$, does not involve any arbitrary choices of parameters before label judgement. 
We mentioned the possibility of setting a parameter $\rho_j$. As in the other algorithms, the closer this term is to $1$, the stronger the prior.

\subsection{Summary}
In this section, we have essentially applied a linear best-fit model to our problem of interest. 
We have shown that the model, despite being constrained, can be optimized with gradient descent on our bipartite graph.

\section{Minimum Cut Formulation}

\subsection{Background}
Purely combinatorial algorithms have some advantages over statistical approximation algorithms: they do not depend on several parameters for accuracy and convergence. They are also easier to prove valid in theory. Unfortunately, many combinatorial optimization problems are intractable on large graphs. Attempts have been make to apply tractable combinatorial problems to clustering \cite{xwang} and multi-way partitioning \cite{kary}. Darling et al. \cite{dar} embed such problems in a general combinatorial data fusion framework.

Recent work has shown that the multi-way partitioning problem can be solved in polynomial time on special graphs \cite{velednitsky2018polynomial} and that a tailored branch-and-bound algorithm can solve instances with tens of thousands of nodes in just seconds \cite{velednitsky2018isolation}.

The model we will use in this section is a special case of what is sometimes called the ``minimum s-excess'' problem (terminology in \cite{hoch}), which in turn is a relaxation of the ``minimum closure'' problem.

\subsection{Objective Function}
Consider first a binary coloring problem (color $c$ or not $c$).
Associate a binary indicator variable $x_i$ with each node $i$ (both left and right) of our bipartite 
graph $G$.  Ultimately, $x_i = 1$ if node $i$ should be colored $c$ and $x_i = 0$ otherwise. Let $R_c$ be the set of right nodes initially labeled $c$. Given penalty parameters $\pi_i^0$ and $\pi_i^1$, we would like to minimize the following objective:

\begin{align*}
	\text{minimize} & & \sum_{R_c}{\pi_i^0 (1 - x_i)} + \sum_{R\setminus R_c}{\pi_i^1 x_i} + \sum_{i\sim j}{1_{\{x_i \neq x_j \}}} \\
	\text{subject to} & & x_i \in \{0, 1\} \\
\end{align*}
If we change $x_i$ from its prior, namely $x_i:=1_{\{i \in R_c \}}$, then we pay a penalty. In the case of the nodes in $R_c$, this means changing them to $0$ and paying penalty $\pi_i^0$. In the case of the right nodes not in $R_c$, this means changing them to $1$ and paying penalty $\pi_i^1$. In the case of the left nodes, there is no prior so no penalty is paid either way.
We pay a unit penalty if two adjacent nodes disagree with each other.

\subsection{Algorithm}
Consider the following construction:

Take the graph $G$ and add a source and sink node, $s$ and $t$ respectively. Draw an edge from $s$ to every $r \in R_c$ with weight $\pi_i^0$. Draw an edge from every $r \in R \setminus R_c$ to $t$ with weight $\pi_i^1$. Call this graph $G_{s,t}$. 

Given a cut in $G_{s,t}$, let vertex sets $S$ and $T$ be components containing the source and sink nodes, respectively. 
We can calculate the price of this cut in $G_{s,t}$ as follows:
\begin{enumerate}
\item
If any node $r \in R_c$ is in the sink set $T$, we must have paid a penalty of $\pi_i^0$ to cut the edge from $s$ to $i$. 
\item
By similar logic, we must have paid a penalty of $\pi_i^1$ for any node $r \in R\setminus R_C$ in the source set $S$. 
\item
We paid a price of $1$ for every other edge that was originally in $G$ and was cut between $S$ and $T$.
\end{enumerate}

The price of the cut is evidently the same as the value of the objective function above when $x_i = 1$ for $i \in S \backslash \{s\}$ and 
$x_i = 0$ for $i \in T \backslash \{t\}$.
Thus, the minimum cut in $G_{s,t}$ gives us the optimal solution to the problem where the only color choices are $c$ or not $c$.

\subsection{Parameters/Wildness}
What is the effect of the value of the penalties $\pi_i^0$ and $\pi_i^1$? 
The larger $\pi_i^0$, the more importance we give to priors (the harder it is to disregard the prior). 
The larger $\pi_i^1$, the more difficult it is to relabel a node to a new color. 
If $\pi_i^0$ is too large, then we necessarily accept the priors. 
If $\pi_i^1$ is too large, then we never assign new values to nodes.
If $\pi_i^0$ or $\pi_i^1$ is too small, then 
we either assign $0$ to all the nodes and pay a penalty of $\sum_{i \in R_c} \pi_i^0$, or else 
we either assign $1$ to all the nodes and pay a penalty of $\sum_{i \in R \backslash R_c} \pi_i^1$, whichever is cheaper.

Experimentally, the best results seem to be obtained when $\pi_i^0$ is slightly larger than $\pi_i^1$.
To gain some intuition for why $\pi_i^0 = \pi_i^1$ does not produce optimal results, consider the case of just two colors.
By symmetry, the two cuts (one for each color) would be compliments of each other. This means that we would be unable to detect wildness.
If $\pi_i^0 \neq \pi_i^1$, then certain nodes ``on the edge'' will be assigned $0$ or $2$ colors, making it possible to detect wilds.
With this in mind, we propose as a heuristic that $\pi_i^0$ be set to three quarters the degree of $i$ and $\pi_i^1$ be set to half the degree of $i$.

We can run this algorithm once for each color and consolidate the results. Most nodes will be in the source set for just one of the $d$ cuts. 
These nodes are assigned to that color, often their initial color.
The remaining nodes are either in zero source sets or more than one source sets. We label these nodes as wild. The ones in zero source sets we can think of as ``true'' wilds, while the ones in multiple source sets are potentially polychromatic border nodes.

\subsection{Summary}
It is rare that a large-scale data mining problem can be solved with a simple combinatorial algorithm. In this case, we have shown that our problem of interest admits an intutive minimum cut formulation. Being purely combinatorial, it does not require us to specify a termination condition or desired degree of accuracy for convergence.


\section{Bayesian Naive Formulation}

\subsection{Background}
Our ``Naive Bayesian'' formulation is motivated by Naive Bayes classification. In general Naive Bayes classification, a parameter is estimated to be in one of several classes based on evidence. Its posterior probability of being a member of each class is calculated and then it is assigned to the class in which it has the highest probability. In this algorithm, we successively apply Naive Bayes classification to one set of nodes based on the other. We do not extract a final classification from the probability distribution until the end.

The procedure of iteratively performing Bayesian updates in a graph context is very well studied and termed Belief Propagation \cite{yedi}. Technically, we are working in a Markov Random Field since our graph contains cycles, but in practice we are iteratively updating one set of nodes in terms of the other (left or right), so we could think of our graph as the union of two Bayesian Networks. There is extensive literature covering Belief Propagation in Bayesian networks, including foundational complexity results \cite{coop} and several books \cite{kja,jen}.

\subsection{Algorithm}

\subsubsection{Motivation}
Consider a single left or right node $v$. 
In order to accept color $c$ as valid for node $v$, consider its neighbors $v'$. One of three properties must hold:
\begin{enumerate}
	\item $v'$ is also colored $c$.
	\item $v'$ is wild color $\mho$.
	\item $v'$ is neither $c$ or $\mho$, but the edge $\{v, v'\}$ is a misattribution.
\end{enumerate}

\subsubsection{Bayes Formula Under Conditional Independence Approximation}
Let $\chi: L \cup R \to C \cup \{\mho\}$ denote a coloring of all nodes, which ideally would coincide with $c:R \to C\cup \{\mho\}$
 on the right nodes, and with $\bar{c}:L \to C\cup \{\mho\}$ on the left nodes. 
In the Bayesian belief-propagation scheme, this unknown map $\chi$
is treated as random, and so we can make statements about the joint distribution of neighboring colors $\{\chi(w),  w \in \Gamma(v)\}$,
for a node $v$, and about the conditional distribution of $\chi(v)$, given $\{\chi(w),  w \in \Gamma(v)\}$.
Our model for the probability distribution of $\chi$ starts with the assumption that ideally edges should exist only
between right and left nodes of the same color, then successively introduces the types of errors and adjusts our scheme to compensate.

The probability $v$ has color $i$, given the colors of its neighbors, can be written using Bayes' formula.
Here $I_w \subset C$ denotes some set of colors, for each node $w$.
\begin{small}
\begin{equation} \label{e:nbf}
P \left(\chi(v)=i \mid \bigcap_{w \in \Gamma(v)} \{ \chi(w) \in I_w  \} \right)
 \propto P \left(\bigcap_{w \in \Gamma(v)} \{ \chi(w) \in I_w  \}  \mid \chi(v)=i \right) P(\chi(v) = i).
\end{equation}
\end{small}
In the following subsections we introduce and progressively update our
scheme for determining the conditional distribution of $\{\chi(w),  w \in \Gamma(v)\}$, given $\chi(v)$.

If our assumption is that adjacent nodes must have the same color, without exception, then our update term becomes
\begin{equation} \label{e:nbfupdate}
P \left(\bigcap_{w \in \Gamma(v)} \{ \chi(w) \in I_w  \}  \mid \chi(v)=i \right) = P \left(\bigcap_{w \in \Gamma(v)} \{ \chi(w) = i \}  \mid \chi(v)=i \right)
\end{equation}

We use the \textbf{conditional independence approximation} to the probability of a monochromatic neighborhood:
\begin{equation} \label{e:cia}
P \left(\bigcap_{w \in \Gamma(v)} \{ \chi(w) = i \}  \mid \chi(v)=i \right) 
= \prod_{w \in \Gamma(v)}{P\left(\chi(w) = i \mid  \chi(v)=i\right)}.
\end{equation}

To illustrate why this is reasonable, consider a node $v$ and its neighbors $w \in \Gamma(v)$. We assume that if $v$ is colored red then all of its neighbors $w$ must be red. The color allocations of these right neighbors are treated as conditionally independent, given the color of $v$.
Therefore the probability that $w_j$ is red for all $w_j \in \Gamma(v)$ equals the product over $j$
of the probability that each individual $w_j$ is red.

\subsubsection{Including Wildness in the Bayesian Formulation}
The next step in the Bayesian algorithm is to account for the probability that a node is assigned the wild color $\mho$.
First assign a color prior $(\pi_0, \pi_1, \dots, \pi_d)$, where $\pi_0$ is the proportion $\omega$ of wild color vertices,
and for $i \geq 1$, $\pi_i$ is the proportion of vertices whose true color is $i$. 

\begin{lemma} \label{l:wildBayes}
Given that $v$ is wild,
that is $\chi(v) = \mho$, the conditional law 
\[
P \left(\bigcap_{w \in \Gamma(v)} \{ \chi(w) = i_w \}  \mid \chi(v) = \mho \right) 
\]
of $\{\chi(w),  w \in \Gamma(v)\}$ takes the form
\begin{equation} \label{e:priorprod}
\prod_{w \in \Gamma(v)} \pi_{i_w}
\end{equation}
for any values $i_w \in \{0, 1, 2, \ldots, d\}$. Take $p^w_i:= P(\chi(w) = i)$ for $i \geq 1$, and $p^w_0:= P(\chi(w) = \mho)$, so
the probability vector $(p^w_0, p^w_1,\ldots,p^w_d)$
encodes the law of $\chi(w)$ for $w \in \Gamma(v)$. 
Under the conditional independence approximation, the conditional probability
that $v$ is wild, given the law of $\chi(w)$ for $w \in \Gamma(v)$, is given by:
\[
P \left( \chi(v) = \mho  \mid \bigcap_{w \in \Gamma(v)} \{ \chi(w) \} \right)
\propto \prod_{w \in \Gamma(v)} \left( 
\sum_{i=0}^d \pi_i p^w_i \right).
\]
\end{lemma}

\begin{proof}
The first assertion restates the definition of a wild color vertex, namely one that gives no information
about the color of its neighbors. The second assertion follows from Bayes rule, from
the conditional independence approximation, and from the first assertion, because when the product of sums
is multiplied out, each assignment of colors to the elements of $\Gamma(v)\}$ receives
 a weight given by (\ref{e:priorprod}).
\end{proof}

\subsubsection{Including Misattribution in the Bayesian Formulation}
Our final improvement to the Bayes probability distribution is to incorporate a small probability that an uninformative edge exists between otherwise correctly labeled nodes. This is intended to reflect the possibility of misattribution in the graph. Assume that the misattribution rate is $\alpha$. 
Consider any edge $e:=\{v, w\}$, and let $E_* \subset E$ denote the set of uninformative edges leading to misattribution. Fix any color $i$.
In the following inclusion-exclusion calculation, two terms are omitted because $\{\chi(w) = i\} \cap \{\chi(w) = \mho \} = \emptyset$.
\begin{equation*}
\begin{split}
& P(\{\chi(w) = i\} \cup \{\chi(w) = \mho \} \cup \{e \in E_*\})\\
& = P(\chi(w) = i) + P(\chi(w) = \mho) + P(e \in E_*)\\
& - P(\{\chi(w) = i\} \cap \{e \in E_*\}) - P\{\chi(w) = \mho \} \cap \{e \in E_*\}).
\end{split}
\end{equation*}
Abbreviate the conditional probability $P(\cdot \mid \chi(v) = i)$ to $P_i(\cdot)$. 
When we condition on the event $\{\chi(v) = i\} $, the conditional independence approximation may be combined with
the last identity to give
\begin{equation*}
\begin{split}
& P_i(\{\chi(w) = i\} \cup \{\chi(w) = \mho \} \cup \{e \in E_*\})\\
& = P_i(\chi(w) = i) + P_i(\chi(w) = \mho) + \\
& P_i(e \in E_*) (1 -  P_i(\chi(w) = i) - P_i(\chi(w) = \mho)).
\end{split}
\end{equation*}
The full generalization of the update term in Bayes' formula (\ref{e:nbfupdate}) includes the possibility that some $w \in \Gamma(v)$ may be wild,
and some edges $\{v,w\}$ may be misattributed. Continuing to write $P(\cdot \mid \chi(v) = i)$ as $P_i(\cdot)$, the joint
law of $\{\chi(w), w \in \Gamma(v)\}$ is expressed in a generalized form of the right side of (\ref{e:cia}) as:
\begin{equation} \label{e:genform}
\begin{split}
& \prod_{w \in \Gamma(v)} \left[
q_{i, w} + (1 - q_{i, w}) P_i(\{v,w\} \in E_*) 
\right];
\\
& q_{i, w}:=P_i(\chi(w) = i) + P_i(\chi(w) = \mho).
\end{split}
\end{equation}

\subsection{Parameters}
Our algorithm requires an estimate for the misattribution rate $\alpha$. It would be unrealistic to use the same $\alpha$ that we used in constructing the graph, since that will be unknown in a general graph problem. Thus, we need to create a model for $\alpha$. Overestimating the true $\alpha$ will make our probability updates less confident than they should be. Underestimating the true $\alpha$ will make our probability updates too confident. Since we would rather be in the former case, we can model $\alpha$ by assuming that all the noise in the graph (SNR) comes from misattribution. In that case, we have $\text{SNR} \approx (1-\alpha)^2$, so our over-estimate for the true $\alpha$ will be $1-\sqrt{\text{SNR}}$. If we are given values for other types of error, we could derive a more accurate $\alpha$ from the SNR.

We acknowledge that it is possible to create a model where the $\alpha$ are edge-specific, but it would be difficult to balance parameters such that we could distinguish wild nodes from nodes with all neighbors misattributed. We prefer to focus on message-passing algorithms in which the nodes are impartial to the sources of their messages.

Aside from $\alpha$, we must define initial priors for the labels on the nodes. On the left nodes, we set these priors uniform. On the right nodes, we have three values to consider: the prior on the given label, the prior on wildness, and the prior on not matching the label. We decide to set these approximately equal, with the prior on wildness slightly discounted and the prior on not maching the label distributed uniformly accross the other labels. This prior might be further improved by weighting each color by its frequency in the graph at large.

\subsection{Summary}
The Naive Bayesian algorithm creates iterative updates of our belief about the left nodes given the right nodes, and vice versa. 
The update scheme follows naturally from the interaction of different sorts of errors, from a local conditional independence assumption, and from applying Bayes' Rule. 
Out of all our algorithms, the Naive Bayesian algorithm models the problem most directly: 
we assign probabilities to certain existing structures and consider the natural probabalistic consequences according to Bayes' Rule.

\section{Harmonic Function Formulation}

\subsection{Background}

\subsubsection{Classical Background}
In classical analysis, the \textbf{Dirichlet problem} seeks to extend
a real-valued continuous function $f$, defined on the boundary $\partial B$ of a bounded closed subset 
$B$ of $\mathbf{R}^d$, 
to a harmonic $f$ on the interior of $B$. Harmonic means that $\Delta f$ = 0, where $\Delta$ is the Laplace operator.
 The probabilistic solution
consists in starting a Brownian motion $(X_t)_{t \geq 0}$ at an arbitrary $x \in B$, and stopping it at the first exit time $\tau$ from $B$.
The expected value of $f(X_{\tau})$ is well defined because $X_{\tau} \in \partial B$; it 
agrees with $f(x)$ when $x \in \partial B$, and is a harmonic function of the starting point $x$, so it solves the Dirichlet problem.

\subsubsection{Random Walk on Graphs}
The construction extends from the case of Brownian motion on $\mathbf{R}^d$ to that of random walks on graphs, as 
described by Lovasz \cite{lov}; here the role of the Laplace second order differential operator is replaced by the Laplacian matrix $L = D - A$,
where rows and columns are indexed by graph vertices, $D$ is the diagonal matrix of vertex degrees, and
$A$ is the graph adjacency matrix.

More generally, suppose $G = (V, E, (w_{i,j})_{\{i,j\} \in E \}})$ is an edge weighted graph, with all edge weights $w_{i,j} > 0$.
Redefine $D$ and $A$ so that
\[
D_{i,i}:= \sum_{j \sim i} w_{i,j} > 0; \quad A_{i,j} = w_{i,j} 1_{\{i \sim j \}}.
\]
This corresponds to a random walk $(X_t)_{t=0, 1, \ldots}$ on the vertex set $V$, where transition probabilities 
from vertex $v$ to its neighboring vertices
are determined by the normalized weights of edges incident to $v$. 

A real-valued vector $f$ indexed by the vertices of the graph is called \textbf{harmonic} if $Lf = 0$ at every vertex,
 which is the same as saying that the value of $f$ at vertex $v$ is the edge-weighted average of its value at adjacent vertices.
For $n \geq 2$, define $f:V \to \mathbf{R}^n$ to be harmonic if every component is harmonic.

\subsubsection{Constructing Harmonic Functions on Graphs}
Suppose the vertex set $V$ of a finite connected weighted graph is partitioned into $V_0 \cup V_1$.
Consider $V_0 \neq \emptyset$ as a boundary, on which
 a function  $g:V_0 \to \mathbf{R}^n$ is defined. We imitate in (\ref{e:hfgraph}) below the
 solution of the Dirichlet problem to construct a harmonic function $f:V \to \mathbf{R}^n$
which agrees with $g$ on $V_0$.

\begin{lemma} \label{l:energy}
Let $G = (V, E,(w_{i,j})_{\{i,j\} \in E})$ be
a finite connected weighted graph. Partition the vertex set $V$ as $V_0 \cup V_1$. 
Given a function  $g:V_0 \to \mathbf{R}^n$ whose components $(g_1(u), \ldots, g_n(u))$ are non-negative, and sum to 1 at each $u \in V_0$,
there is a unique function $f:V \to \mathbf{R}^n$, all of whose components are harmonic, such that $f_{|V_0} = g$. Furthermore:
\begin{enumerate}
\item
The components $(f_1(u), \ldots, f_n(u))$ are non-negative and sum to 1 at every $u \in V$.
\item
The energy $\mathcal{E}(f)$ is minimum  among all functions which agree with $g$ on $V_0$, where
\begin{equation} \label{e:energy}
\mathcal{E}(f):= \frac{1}{2} \sum_{i \in V} \sum_{j \sim i} w_{i,j} \|f(i) - f(j)\|^2.
\end{equation}
\end{enumerate}
\end{lemma}

\begin{proof}
Since $G$ is finite and connected, random walk on $V$ is positive recurrent. Consider the random walk $(X_t)_{t=0, 1, \ldots}$,
stopped at the first step $\tau$ when it reaches a vertex in $V_0$. By positive recurrence, $\tau$ is finite with probability 1, and
$\tau = 0$ when the random walk starts in $V_0$. Hence $g(X_{\tau})$ is a well defined random variable.
For any vertex $u \in V$, define 
\begin{equation} \label{e:hfgraph}
f(u):=\mathbf{E}[g(X_{\tau}) \mid X_0 = u].
\end{equation}
This definition transfers from $g$ to $f$ the 
non-negativity of each component, and the fact that the components sum to 1.
By conditioning on the value of $X_1$, 
and invoking the Markov property, we see that for every component $f_j$
\[
f_j(u) = \sum_{v \sim u} \frac{w_{u,v}}{D_{u,u}} \mathbf{E}[g_j(X_{\tau}) \mid X_1 = v]
=  \sum_{v \sim u} \frac{A_{u,v}}{D_{u,u}} f_j(v).
\]
Therefore $D f_j - A f_j  = 0$ for every component $j$, and $f_j$ agrees with $g_j$ on $V_0$. So $f$ is harmonic.
Uniqueness of the solution to $(D - A)f = 0$, $f_{|V_0} = g$, follows from linear algebra.

It remains to check that $f$ is energy-minimizing.
Consider a perturbation $f + h$ for $h:V \to \mathbf{R}$ with $h_{|V_0} = 0$. A routine computation shows that
\[
\mathcal{E}(f+h) - \mathcal{E}(f) = 
2 \sum_i h(i) (D_{i,i} f(i) - \sum_{j \sim i} A_{i,j} f(j)) + O(\|h\|^2).
\]
Therefore the derivative $\nabla f (h)$ is given by
\[
\nabla f (h) = 2 \sum_i h(i) Lf(i).
\]
Since $Lf(i) = 0$ for $i \in V_1$, the right side must be zero for all perturbations $f + h$ with $h_{|V_0} = 0$.
Hence $f$ has minimum or maximum energy. A second derivative analysis shows the energy is minimum.
\end{proof}

\subsection{Objective Function}

\subsubsection{Symmetrization}
As Markov processes, Brownian motion on $\mathbf{R}^d$ and harmonic function on a weighted graph differ in that the transition
function for the former is symmetric (the rate of transition from $x$ to $y$ in time $t$  is the same as from $y$ to $x$), 
whereas for the latter it is typically asymmetric. For example,
if adjacent vertices $v, v'$ have degrees $d, d'$ with $d < d'$, then the unweighted transition from $v$ to $v'$ in one step occurs with
probability $1/d$, while the transition from $v'$ to $v$ in one step occurs with probability $1/d' < 1/d$. Some authors
replace the Laplacian matrix $L = D - A$ by the symmetric normalized Laplacian matrix $L^s:= I - D^{-1/2} A D^{-1/2}$.
The meaning of the term harmonic function changes accordingly. In our application we do not symmetrize the Laplacian.
For other examples of the use of harmonic functions for label propagation, see Fouss et al \cite{fou}.

\subsubsection{Harmonic Functions For Label Propagation}
If we have $d$ colors, we are concerned with functions from the vertices of the auxiliary graph $G'$ 
(see \ref{graph_expression}) to $\mathbf{R}^d$.
The auxiliary nodes will play the role of the boundary $V_0$ in Lemma \ref{l:energy}.
The value of the function at right node $v_k$ will be expressed by a $d$-dimensional vector of probabilities
$r^k = (r_1^k,\ldots,r_d^k)$, where $r_i^k$ is the probability that $v_k$ has color $i$. 
Naturally, $\sum_i{r_i^k} = 1$. We define $\ell^j$ similarly for $u_j \in L$. 
We encode the initial coloring of $v_k$ by tagging the auxiliary node $w_k \sim v_k$ with a vector $h^k$ which equals basis vector $e_i$
in the case where color $i$ is initially assigned to right vertex $v_k$.
The energy (\ref{e:energy}) takes the form
\begin{equation} \label{e:energybpg}
 \sum_{\ell^j \sim r^k}{||\ell^j-r^k||^2} + \theta \sum_{v^k \in R}{||r^k-h^k||^2}.
\end{equation}
To minimize this energy means to find the best compromise between (a) agreement in color vectors of adjacent left and right nodes, and
(b) agreement between color vectors of right nodes and their initial colorings. The penalty parameter $\theta > 0$ 
defines the relative weight of (b) versus (a).

Lemma \ref{l:energy} asserts that a function on $L \cup R$ which minimizes this energy, given the boundary
conditions on the auxiliary nodes, is harmonic on $L \cup R$ in each component, with respect to a weighting on the edges
determined by $\theta$. This weighting is most easily expressed in terms of the followng transition probabilities for the random walk.

\subsection{Algorithm}

\subsubsection{Random Walk Transition Probabilities}
We consider a transient random walk on the auxiliary graph $G'$ (see \ref{graph_expression}) defined as follows:
\begin{itemize}
	\item The auxiliary nodes are the absorbing states.
	\item When we are at a right node, we move to its adjacent auxiliary node with probability 
$p = \theta/(1 + \theta)$. With probability $1-p$, we move to a left neighbor picked uniformly at random.
	\item When we are at a left node, we move to a right neighbor picked uniformly at random.
\end{itemize}

For node $v_i$, let the probabilities of absorption at an auxuiliary node with color $j$ for a walk starting at $v_i$ be $\varphi_i(j)$. 
Then $v_i \mapsto (\varphi_i(1), \ldots, \varphi_i(d))$ is the unique harmonic function described in Lemma \ref{l:energy},
and  minimizes the energy (\ref{e:energybpg}).
The $j$ for which $\varphi_i(j)$ is maximum will be called a \textbf{maximum likelihood color}.

\subsubsection{Message-Passing Algorithm for Harmonic Function}
A harmonic function on a graph may be computed exactly by linear algebra techniques,
whose work scales according to $|V| \cdot |E|$. 
Instead we opt for a message-passing approximation which scales according to $|E|$ and can be easily parallelized.
We state the harmonic function algorithm in terms of quantities $\varphi^{(n)}_i(j)$, for $n = 0, 1, 2, \dots$, defined as follows,
and interpreted in Lemma \ref{rwconverges}. 

Initialize $\varphi^{(0)}_i(j) = 0$ if $v_i$ is a right or left node.
At odd numbered steps $n = 1, 3, 5, \ldots$,
the $ r^k:=\varphi^{(n)}_k(j)$ will be updated for right nodes $v_k$. 
At even numbered steps $n = 2, 4, 6, \ldots$, $\ell^i=\varphi^{(n)}_i(j)$ will be updated for left nodes $v_i$.

Update node $v_i$ at step $n \geq 1$ so that $\varphi^{(n)}_i(j)$ is the weighted average of the $\varphi^{(n-1)}_{i'}(j)$ 
for neighbors $i' \in \Gamma(i)$. For a left node, all right neighbors are weighted equally. For a right node, 
its auxiliary neighbor has weight $p$ and each left neighbor has weight $1-p$ divided by the number of left neighbors.

\begin{lemma} \label{rwconverges}
$\varphi^{(n)}_i(j)$ is the probability that a random walk starting at $v_i$ will be absorbed 
at an auxiliary vertex with color $j$ after at most $k$ steps, and
converges from below to $\varphi_i(j)$ as $n \to \infty$. This convergence
occurs exponentially fast, in the sense that
\[
\sum_{j=1}^d (\varphi_i(j) - \varphi^{(2n)}_i(j) ) \leq (1-p)^n.
\]
\end{lemma}

\begin{proof}
The assertion about $\varphi^{(n)}_i(j)$  is true when $n = 0$, because the right and left vertices are not absorbing.
It holds for all $n$ follows by induction using the Markov property and the transition probabilities for the random walk.

Let $\tau$ denote the first $n$ for which the random walk $X_n$, started
at a left vertex $v_i$, is absorbed at an auxiliary vertex. Then $\tau \in \{2, 4, 6, \ldots\}$,
and $P(\tau = 2 n) = p (1 - p)^n$ by the renewal property of Markov chains. Let $F_j$ denote the event
that final absorption of the random walk occurs in an auxiliary vertex of color $j$. Then
\[
\sum_{j=1}^d \varphi_i(j) = P(\bigcup_j F_j \cap \{\tau \leq 2 n\} \mid X_0 = v_i) +
P(\tau \geq 2 n + 2 \mid X_0 = v_i).
\]
The first summand on the right is $\sum_{j=1}^d \varphi^{(2n)}_i(j)$, and the second is $P(\tau/2 \geq n + 1) = (1 - p)^{n+1}$.
This holds for every $j = 1, 2, \ldots , d$, which implies the inequality in the case where $v_i$ is a left vertex.
The case of a right vertex is similar, except $\tau$ takes odd integer values and $P((\tau-1)/2 \geq n) = (1 - p)^{n}$.
\end{proof}


\subsection{Wildness/Parameters}

\subsubsection{Wildness via Jensen-Shannon Divergence} \label{JSDwildness}
Once we have a reasonable approximation to $\varphi_i(j)$ for a given $i$ for all $j$, we would like to know whether to give the node $i$ its maximum liklihood color or classify it as wild. For this, we introduce the Jensen-Shannon Divergence (JSD), a symmetrized version of the Kullback-Leibler (KL) divergence which is a distance metric for probability distributions. For two probability distributions $P$ and $Q$, we define 
\[
\text{JSD}(P||Q) = \frac{1}{2} D(P||M) + \frac{1}{2} D(Q||M),
\]
where $D$ is the Kullback-Leibler divergence and $M = \frac{1}{2} (P + Q)$. 

When the base of the logarithm in the KL divergence is $2$, the Jensen-Shannon Divergence gives a value between $0$ and $1$. If we let $P_i$ be the 
distribution $(\varphi_i(j))_{1 \leq j \leq d}$ over colors and $Q$ be the probability distribution over colors of labels in the graph, then $\text{JSD}(P_i||Q)$ is close to 1 when node $i$ is normal, and close to zero when it is wild.
For if $P_i = Q$, then the probability that node $i$ is color $j$ equals the probability that a random node is color $j$: $\text{JSD}(P_i||Q)=0$ and we believe $i$ is wild. On the other hand, if $\varphi_i(j) \approx 1$ and $j$ is a rare color in the graph:$\text{JSD}(P_i||Q)\approx 1$ and we are confident that $i$ is normal. We set a parameter $\tau$. If the JSD exceeds $\tau$, we give the node its maximum liklihood label. Otherwise, we give it a wild label.

\subsubsection{Metaparameter Settings}
The value of $p$ affects the expected length of the random walk independent of graph structure. The number of visits to right nodes before absorption has a Geometric distribution, the same as the number of Bernoulli$(p)$ trials needed for the first success. Hence its expected value is $\frac{1}{p}$. The probability that the random walk makes at least $n+1$ visits to right nodes ($2n$ steps) is $(1-p)^n$. 
The larger $p$, the more strength we give to our prior. The smaller $p$, the more diffusion we allow. Experimentally, $p = \frac{1}{12}$ appeared to be a good setting. This was approximately the reciprocal of the mean vertex degree.

Finally, our setting for $\tau$ is entirely experimental. If $\tau$ is small then we are likely to classify some wilds as a color. If $\tau$ is too large, then we are likely to assign too many wild labels.

\subsection{Similarity to Semidefinite Programming Formulation}
Compare the objective function (\ref{eq:sdpvoting}) for the SDP formulation with the energy function (\ref{e:energybpg})
which our harmonic function minimizes. The same cross product terms $\ell^j \cdot r^k$ and $h^k \cdot r^k$ appear in both.
Also the vectors $\ell^j$ and $r^k$ are constrained to be probability vectors in both cases. 
The difference is that terms of the form $\ell^j \cdot \ell^j$ and $r^k \cdot r^k$ appear in (\ref{e:energybpg}),
but not in (\ref{eq:sdpvoting}). Thus the harmonic function also seeks probability vectors with smaller sums of squares,
meaning that they are closer to uniform.

From what we have seen so far, this makes sense: the SDP solution above could be made to have mostly integer entries by a linearity argument. The same is not true of the harmonic function. 

\subsection{Summary}
Random walks on graphs are extremely well studied, as they form the basis of many practical graph algorithms for large-scale clustering, detection, and labeling. In our problem instance, our key observation is understanding how to describe the harmonic function in terms of an appropriate auxiliary graph. 
On a bipartite graph, we demonstrate exponentially fast convergence, without resorting to spectral methods, from a message-passing algorithm
which scales better than linear algebra methods.

\section{Bayesian Multinomial Formulation}

\subsection{Background}
In the Naive Bayes formulation, each vertex was assumed to carry a single unknown color, about which our beliefs were updated.
Motivated by the literature in Monte Carlo methods, Expectation-Maximization, and Variational Bayesian methods \cite{jordan},
we now investigate probabilistic models where each vertex carries a multinomial distribution across the $d$ colors, and
message-passing updates the parameters of these multinomial distributions.
We attempt to combine ideas used in the Naive Bayesian and Harmonic Function models. 

\subsection{Algorithm}
\subsubsection{Multinomial Distribution}

Recall the use of the $\text{Dirichlet}(\alpha_1, \alpha_2, \ldots, \alpha_d)$ distribution as a conjugate prior for the 
parameters of the $\text{multinomial}(p_1, p_2, \ldots, p_d)$. At the point $(p_1, p_2, \ldots, p_d)$ in a $(d-1)$-dimensional
simplex, this Dirichlet density takes value
\[
\frac{1}{B(\alpha)} \prod_{i=1}^d p_i^{\alpha_i - 1},
\]
where $B(\alpha)$ is the multivariate Beta function. In cases where some, but not all, of the $\alpha_i$ are zero,
we obtain in effect a density over a lower dimensional simplex.
Suppose we have assigned this prior and a multinomial sample of size $n$
yields a vector $(n_1, n_2, \ldots, n_d)$ of outcomes for each of the $d$ categories, with $\sum n_i = n$. The posterior
distribution for $(p_1, p_2, \ldots, p_d)$ is also Dirichlet, with parameters $(\alpha_1+n_1, \alpha_2 + n_2, \ldots, \alpha_d + n_d)$.

\subsubsection{Randomized Message-Passing}
Model the color state at left or right node $v_i$ as $\text{multinomial}(p_1^i, p_2^i, \ldots, p_d^i)$. The unknown parameter vector
is assumed to have a $\text{Dirichlet}(\alpha_1^i, \alpha_2^i, \ldots, \alpha_d^i)$ distribution.
At the outset, these Dirichlet parameters are all zero at a left vertex; at a right vertex $\alpha_j^i = \mu_i$ if $v_i$  is initially
colored $j$, and zero otherwise. Here $\mu_i > 0$ is a parameter, which was a quarter of the degree of $v_i$
in our experiments. Zero values of $\alpha_k^i$ will not cause a problem, because
after step $1$ the sum of the parameters at every left vertex will be positive. 

At iterations $n = 1, 3, 5, \dots$ of the message passing algorithm, all the right nodes pass randomized messages to their left neighbors;
at iterations $n = 2, 4, 6, \dots$, all the left nodes pass randomized messages to their right neighbors, and a right node also receives
a deterministic reinforcement vector $ \lambda_i e_j$ if the initial coloring is $j$. Here $e_j$ denotes the $j$-th basis vector in $\mathbf{R}^d$,
and $\lambda_i > 0$ is a parameter, which was an eighth of the degree of $v_i$ in our experiments.

To construct a randomized message to each of its neighbors, node $v_i$ performs one $\text{multinomial}(p_1^i, p_2^i, \ldots, p_d^i)$ trial, 
and if the outcome is of category (color) $j$, the unit vector $e_j$ is sent to every neighbor. The probability that node $v_i$
sends unit vector $e_j$ is 
\[
\frac{\alpha_j^i}{\sum_k \alpha_k^i}
\]
When node $v_k$ receives the unit vector $e_j$, its updated Dirichlet distribution has parameters
$(\alpha_1^k, \alpha_2^k, \ldots, \alpha_d^k) + e_j$. In other words, the $j$-th component increases by $1$.

\subsubsection{Convergence}
The convergence condition for the Bayesian Multinomial is analogous to the definition of a harmonic function: 
it should be the case at every node $v_i$ that if a neighbour $v_{i'} \sim v_i$ is sampled uniformly at random,
and if that neighbor generates a random message, then the probability that the message is $e_j$ coincides (to some
precision) with the probability that $v_i$ itself generates message $e_j$. 
In other words, the algorithm has a fixed point where
\[
\frac{\alpha_j^i}{\sum_k \alpha_k^i} \approx \frac{1}{|\Gamma(v_i)|} \sum_{v_{i'} \sim v_i} \frac{\alpha_j^{i'}}{\sum_k \alpha_k^{i'}},
\quad j \in \{1, 2, \ldots, d\}.
\]
Further work is needed, possibly using martingale analysis of a multiple Polya urn model, in order to prove that such a fixed point
exists for given parameter choices $(\mu_i), (\lambda_i)$.

\subsection{Wildness}
A key advantage of the randomized multinomial message passing algorithm is its ability to incorporate an explicit model for wildness. 
It can be modified so that, at each step, instead of sending a message with a color a node can instead send a message that says it is wild. This happens with higher probability the closer that the Jensen-Shannon Divergence is to zero (see \ref{JSDwildness}).

\subsection{Summary}
Inspired by the Naive Bayesian algorithm, we developed a Bayesian algorithm which more closely resembles expectation maximization: node coloring is assumed to be randomized multinomial, and our characterization of the multinomial parameters by a Dirichlet distribution is updated at each iteration. The termination condition of this algorithm is analogous to that of the harmonic function construction. Its downside is that, due to random sampling, it reaches its conclusion slower than the harmonic function. However, the description in terms of a randomized multinomial gives a more sophisticated technique for broadcasting wildness during the algorithm, instead of determining it retroactively (as in the harmonic function).

\section{Machine Learning and Tensorflow}

\subsection{Background}
It is possible to view this problem as a purely machine-learning problem. For each of the $|R|$ right nodes, we can create an $|L|$-dimensional binary feature vector. For a given right node, each coordinate in the $|L|$-dimensional vector is $1$ if the right node is adjacent to that left node and $0$ otherwise. Our goal, in machine learning terms, is to detect clusters (colors) and outliers (wilds) simultaneously.

Typical clustering algorithms face some challenges in this context. First, the labels may not necessarily define clusters but could define a union of sub-clusters. This means that the true number of graph clusters could be significantly larger than the number of labels if the labels represent the union of smaller categories. Second, we are in an extremely large dimension with extreme sparsity where the notion of proximity for classification is less meaningful \cite{agg}. This essentially rules out algorithms such as $k$-means, and casts doubt on SVMs or random forests.

For the reasons mentioned above, we opt to use neural networks. We are further motivated by the fact that the left nodes intuitively seem to behave somewhat like neurons. Being adjacent to a particular left node is ultimately evidence for or against a particular color in a potentially non-linear way. For example, being adjacent to a particular pair of different colored nodes may be strong evidence of wildness, or relatively weak evidence of wildness depending on the pair (if they are ``distant,'' it is strong evidence, while if they are ``nearby'' it is weak evidence).

Though it is a somewhat atypical application, neural networks for sparse classification problems have previous been considered in network problems \cite{tan} and in recommender systems \cite{goo2}.

\subsection{Transductive Learning}
The machine learning context applicable here is called ``transductive learning.'' In general, ``transductive learning'' is learning where the training and testing sets are subsets of the same data set. Learning uses direct deductions rather than surmising general principles. In this specific case, our training and testing set are the same, but the training set contains some errors. We would like to predict the correct testing labels given the sometimes incorrect training labels. Arguably, it is a semi-supervised learning problem. We think of the set of left node adjacencies as features on the right nodes and right node colors as the labels to be predicted from those features. Our goal is to predict the color of a right node given the left nodes that it neighbors.

\subsection{Objective Function}
We will create a simple neural network with a single intermediate layer. Our layers are the input layer, the fully connected intermediate layer, and the output layer. The input layer will always have size $|L|$ and the output layer will always have size $|C|$. Without any additional a priori assumptions about structure, it is typical to set the intermediate layer to half the size of the input layer. To give ourselves extra flexibility, we will consider an intermediate layer with $(3|L| + 2|C|)/5$ neurons.

In optimization terms, we are finding the optimal weights for a large objective function. 
Let $A_j$ be an $|L|\times 1$ input vector in the training set, which is a column of the adjacency matrix. Let $W_1$ be the $|L|\times|M|$ set of weights between the input and intermediate layer, where $|M|$ is the size of the intermediate layer, and let $B_1$ be the $|M|\times 1$ bias vector for the intermediate layer. Finally, let \texttt{relu} be the activation function typically used in neural networks applied component-wise and let $W_2$ be the $|M|\times|C|$ set of weights between the intermediate and final layer. Our prediction for this particular column $A_j$ can be expressed as a probability vector:
$$Y_j = \text{softmax}[\texttt{relu}(A_j^T W_1+B_1^T)W_2]$$
Our prior for $Y_j$ can be expressed as a unit vector $Z_j$. We would like to minimize the sum of the cross entropy between $Y_j$ and the ``true'' vectors $Z_j$. Formally, we could write:
$$\min{\sum_{j}{Y_j^T -\log{(Z_j)}}}$$
where the logorithm is taken component-wise. Our neural network will tend towards a $W_1$, $B_1$, and $W_2$ that minimize this objective.

\subsection{Algorithm}
We decided to implement our simple Neural Network using Google's TensorFlow, an open-source Library designed to streamline the construction and testing of neural networks optimized for GPU computations \cite{goo1}. TensorFlow has been shown to provide leverage in diverse learning tasks where attempts are made to infer additional unseen relations between the features \cite{goo2}.

Since we are working in a transductive learning context, we perform several iterations of the following (in our case, $16$ iterations): divide the set into half training data and half testing data, predict the results of the testing data from the training data, repeat. At each iteration, we expect to have $\frac{|R|}{2}$ training samples, each a vector of length $|L|$. We optimize the weights on the connections until the cross entropy for the color predictions is less than $0.20$. We record the prediction at each of the testing nodes.

\subsection{Parameters for Color and Wildness}
After performing $t$ iterations of our algorithm, each right node has an average of $\frac{t}{2}$ labels (with multiplicity). We consider these labels. We would like to give a small bias towards agreeing with the prior, so we assign the node to its prior if the latter agrees with at least $\frac{1}{3}$ of the labelings. If that is not the case, we consider the most frequent label. If that label is present at least $\frac{2}{3}$ of the time, we assign the node that color. If neither condition holds, we mark the node wild. We chose to adjust these values from a default value of $\frac{1}{2}$ (just taking the majority vote) when we noticed that the algorithm was over-zealous in deciding to relabel nodes. It is important to note that when a node is in the testing set, the neural network does not know its prior. So, in a sense, the prior of a node is never used in its own classification (only in the classification of its neighbors).

\subsection{Summary}
In this section, we established that our problem could be treated as a high-dimensional, sparse clustering problem, but also explained some potential limitations of treating it with this mindset. We further argued that, given these limitations, a machine learning framework that might work for our problem is a neural network. Though we have not suggested any deviations from standard machine learning literature, we include a neural network for the sake of comparison.

\section{Performance}

\subsection{Parameters and Metrics}

We wish to understand the relative performance of our algorithms on many different synthetic graphs. Our graph for testing will be generated by either the circle model or the power law model. Our small graphs will consist of $5100$ left nodes, $1700$ right nodes, and $70$ colors. Our large graph will consist of five times as many nodes and colors. We test on seven different randomly generated instances, summarized in the tables below:

\begin{center}
\begin{tabular}{|l|cc|}
\hline
Parameter & Small & Large \\
\hline
Colors & 70 & 350 \\
Left & 5100 & 25500 \\
Right & 1700 & 8500 \\
\hline
\end{tabular}
\end{center}

\begin{center}
\begin{tabular}{|l|cccc|}
\hline
Instance & Model & Size & Misclass & Wildness \\
\hline
1 & Circle & Small & 0.15 & 0.15 \\
2 & Circle & Large & 0.15 & 0.15 \\
3 & Power & Small & 0.05 & 0.05 \\
4 & Power & Small & 0.05 & 0.15 \\
5 & Power & Small & 0.15 & 0.05 \\
6 & Power & Small & 0.15 & 0.15 \\
7 & Power & Large & 0.15 & 0.15 \\
\hline
\end{tabular}
\end{center}

In testing our algorithms, we wish to understand the extent to which we can trust their results. The output of our algorithms can be one of three judgements: wild, same as prior, or misclassified. We will study the fraction of nodes which fall into each true category given their classified category. When a node is classified as prior or wild, we see how many are in the true categories normal, wild, and misclassified. When we are dealing with misclassifications, we need to distinguish between cases where the misclassification was corrected and cases where is was detected but ``miscorrected'' (corrected to the wrong value). We summarized our fields in the following table:

\begin{center}
\begin{tabular}{|l|l|}
\hline
A & Algorithm Name \\
T & Time Elapsed in Seconds \\
\hline
W & Classified Wild Count \\
W:W & True Wild given Wild Class \\
M:W & True Misclass given Wild Class \\
N:W & True Normal given Wild Class \\
\hline
P & Classified Prior Count \\
N:P & True Normal given Prior Class \\
M:P & True Misclass given Prior Class \\
W:P & True Wild given Prior Class \\
\hline
R & Classified Relabeled Count \\
C:R & True Misclass given Relabled Class, Relabel Correct \\
M:R & True Misclass given Relabled Class, Relabel Wrong \\
W:R & True Wild given Relabeled Class \\
N:R & True Normal given Relabeled Class \\
\hline
Wk & Weak Correctness \\
Str & Strong Correctness \\
\hline
\end{tabular}
\end{center}

Notice that we have two ultimate notions of correctness. Strong correctness is the strictest metric. We only get credit for a node if we either give it the correct label or correctly identify it as wild. Weak correctness is a relaxation: we consider a node anomalous if it is either misclassified or wild. We get credit for a node if we correctly detect that it is anomalous, even if we are not able to correctly identify the type of anomaly.

As a check, we expect that the sum of the values conditioned on a fixed classification will be $1$. In the tables below, we express our probability metrics over the seven different data sets (rounded to two decimal places). We have not made considerable effort to optimize performance, so runtimes should be used to measure scaling and not absolute performance.

\subsection{Raw Data}
The raw data is presented in seven tables.

\begin{landscape}


\begin{table} 
\label{r1}
\begin{footnotesize}
\begin{tabular}{|llr|rrrr|rrrr|rrrrr|rr|}
\hline
T1 &      A &      T &    W &   W:W &   M:W &   N:W &     P &   N:P &   M:P &   W:P &    R &   C:R &   M:R &   W:R &   N:R &    Wk &   Str \\
\hline
0 &   TRV. &    0 &    0 &  0.00 &  0.00 &  0.00 &  1700 &  0.73 &  0.13 &  0.14 &    0 &  0.00 &  0.00 &  0.00 &  0.00 &  0.73 &  0.73 \\
1 &   VOT. &    1 &  200 &  1.00 &  0.00 &  0.00 &  1304 &  0.94 &  0.03 &  0.03 &  196 &  0.95 &  0.03 &  0.01 &  0.02 &  0.95 &  0.95 \\
2 &   GRD. &  104 &  255 &  0.82 &  0.13 &  0.05 &  1253 &  0.97 &  0.01 &  0.02 &  192 &  0.94 &  0.01 &  0.00 &  0.05 &  0.96 &  0.94 \\
3 &   CUT. &   31 &  340 &  0.69 &  0.24 &  0.07 &  1219 &  0.99 &  0.00 &  0.01 &  141 &  0.96 &  0.01 &  0.00 &  0.03 &  0.98 &  0.93 \\
4 &  N.BA. &   11 &  239 &  1.00 &  0.00 &  0.00 &  1215 &  1.00 &  0.00 &  0.00 &  246 &  0.87 &  0.03 &  0.00 &  0.10 &  0.98 &  0.98 \\
5 &  H.FN. &  118 &  255 &  0.93 &  0.04 &  0.03 &  1241 &  0.99 &  0.01 &  0.00 &  204 &  0.95 &  0.04 &  0.00 &  0.01 &  0.99 &  0.97 \\
6 &  M.BA. &  238 &  255 &  0.93 &  0.03 &  0.04 &  1223 &  0.99 &  0.01 &  0.00 &  222 &  0.90 &  0.04 &  0.00 &  0.06 &  0.98 &  0.97 \\
7 &   N.N. &  468 &  328 &  0.73 &  0.17 &  0.10 &  1190 &  0.98 &  0.01 &  0.01 &  182 &  0.81 &  0.04 &  0.00 &  0.14 &  0.95 &  0.91 \\
\hline
\end{tabular}
\end{footnotesize}
\end{table}

\begin{table} 
\label{r2}
\begin{footnotesize}
\begin{tabular}{|llr|rrrr|rrrr|rrrrr|rr|}
\hline
T2 &      A &       T &     W &   W:W &   M:W &   N:W &     P &   N:P &   M:P &   W:P &     R &   C:R &   M:R &  W:R &   N:R &    Wk &   Str \\
\hline
0 &   TRV. &     0 &     0 &  0.00 &  0.00 &  0.00 &  8500 &  0.71 &  0.14 &  0.15 &     0 &  0.00 &  0.00 & 0.00 &  0.00 &  0.71 &  0.71 \\
1 &   VOT. &     4 &  1129 &  1.00 &  0.00 &  0.00 &  6317 &  0.95 &  0.03 &  0.02 &  1054 &  0.90 &  0.04 & 0.00 &  0.05 &  0.96 &  0.95 \\
2 &   GRD. &  2455 &  1275 &  0.79 &  0.15 &  0.06 &  6246 &  0.95 &  0.01 &  0.04 &   979 &  0.92 &  0.03 & 0.00 &  0.05 &  0.95 &  0.92 \\
3 &   CUT. &   971 &  1796 &  0.71 &  0.22 &  0.07 &  5934 &  1.00 &  0.00 &  0.00 &   770 &  0.98 &  0.01 & 0.00 &  0.02 &  0.98 &  0.94 \\
4 &  N.BA. &   255 &  1282 &  1.00 &  0.00 &  0.00 &  5928 &  1.00 &  0.00 &  0.00 &  1290 &  0.84 &  0.05 & 0.00 &  0.11 &  0.98 &  0.97 \\
5 &  H.FN. &  3294 &  1275 &  0.86 &  0.04 &  0.10 &  6163 &  0.96 &  0.01 &  0.03 &  1062 &  0.93 &  0.05 & 0.00 &  0.02 &  0.95 &  0.94 \\
6 &  M.BA. &  5052 &  1275 &  0.91 &  0.02 &  0.07 &  6133 &  0.97 &  0.02 &  0.02 &  1092 &  0.91 &  0.05 & 0.00 &  0.04 &  0.96 &  0.95 \\
\hline
\end{tabular}
\end{footnotesize}
\end{table}

\begin{table} 
\label{r3}
\begin{footnotesize}
\begin{tabular}{|llr|rrrr|rrrr|rrrrr|rr|}
\hline
T3 &      A &      T &   W &   W:W &   M:W &   N:W &     P &   N:P &   M:P &   W:P &    R &   C:R &   M:R &   W:R &   N:R &    Wk &   Str \\
\hline
0 &   TRV. &    0 &   0 &  0.00 &  0.00 &  0.00 &  1700 &  0.90 &  0.04 &  0.06 &    0 &  0.00 &  0.00 &  0.00 &  0.00 &  0.90 &  0.90 \\
1 &   VOT. &    1 &  66 &  1.00 &  0.00 &  0.00 &  1557 &  0.98 &  0.01 &  0.01 &   77 &  0.77 &  0.00 &  0.13 &  0.10 &  0.97 &  0.97 \\
2 &   GRD. &   66 &  85 &  0.61 &  0.19 &  0.20 &  1579 &  0.95 &  0.02 &  0.03 &   36 &  0.89 &  0.00 &  0.00 &  0.11 &  0.95 &  0.94 \\
3 &   CUT. &   19 &  94 &  0.73 &  0.07 &  0.19 &  1533 &  0.98 &  0.00 &  0.02 &   73 &  0.89 &  0.00 &  0.03 &  0.08 &  0.97 &  0.96 \\
4 &  N.BA. &    6 &  94 &  0.97 &  0.00 &  0.03 &  1518 &  1.00 &  0.00 &  0.00 &   88 &  0.81 &  0.01 &  0.02 &  0.16 &  0.99 &  0.98 \\
5 &  H.FN. &   69 &  85 &  0.82 &  0.01 &  0.16 &  1508 &  0.99 &  0.00 &  0.00 &  107 &  0.64 &  0.01 &  0.19 &  0.16 &  0.98 &  0.96 \\
6 &  M.BA. &  156 &  85 &  0.85 &  0.01 &  0.14 &  1513 &  0.99 &  0.00 &  0.01 &  102 &  0.70 &  0.00 &  0.15 &  0.16 &  0.98 &  0.97 \\
\hline
\end{tabular}
\end{footnotesize}
\end{table}

\begin{table} 
\label{r4}
\begin{footnotesize}
\begin{tabular}{|llr|rrrr|rrrr|rrrrr|rr|}
\hline
T4 &      A &      T &    W &   W:W &   M:W &   N:W &     P &   N:P &   M:P &   W:P &    R &   C:R &   M:R &   W:R &   N:R &    Wk &   Str \\
\hline
0 &   TRV. &    0 &    0 &  0.00 &  0.00 &  0.00 &  1700 &  0.82 &  0.04 &  0.14 &    0 &  0.00 &  0.00 &  0.00 &  0.00 &  0.82 &  0.82 \\
1 &   VOT. &    1 &  118 &  1.00 &  0.00 &  0.00 &  1474 &  0.93 &  0.01 &  0.05 &  108 &  0.43 &  0.01 &  0.44 &  0.12 &  0.94 &  0.91 \\
2 &   GRD. &   66 &  255 &  0.62 &  0.14 &  0.24 &  1431 &  0.93 &  0.01 &  0.06 &   14 &  0.71 &  0.00 &  0.07 &  0.21 &  0.90 &  0.88 \\
3 &   CUT. &   19 &  228 &  0.82 &  0.07 &  0.11 &  1409 &  0.96 &  0.00 &  0.04 &   63 &  0.68 &  0.02 &  0.14 &  0.16 &  0.95 &  0.93 \\
4 &  N.BA. &    7 &  229 &  0.99 &  0.00 &  0.01 &  1364 &  1.00 &  0.00 &  0.00 &  107 &  0.55 &  0.03 &  0.17 &  0.25 &  0.98 &  0.97 \\
5 &  H.FN. &   75 &  255 &  0.82 &  0.06 &  0.13 &  1346 &  0.99 &  0.01 &  0.01 &   99 &  0.38 &  0.02 &  0.29 &  0.30 &  0.95 &  0.93 \\
6 &  M.BA. &  147 &  255 &  0.82 &  0.05 &  0.13 &  1348 &  0.99 &  0.00 &  0.00 &   97 &  0.45 &  0.02 &  0.31 &  0.22 &  0.96 &  0.94 \\
\hline
\end{tabular}
\end{footnotesize}
\end{table}

\begin{table} 
\label{r5}
\begin{footnotesize}
\begin{tabular}{|llr|rrrr|rrrr|rrrrr|rr|}
\hline
T5 &      A &      T &    W &   W:W &   M:W &   N:W &     P &   N:P &   M:P &   W:P &    R &   C:R &  M:R &   W:R &   N:R &    Wk &   Str \\
\hline
0 &   TRV. &    0 &    0 &  0.00 &  0.00 &  0.00 &  1700 &  0.80 &  0.15 &  0.06 &    0 &  0.00 &  0.00 &  0.00 &  0.00 &  0.80 &  0.80 \\
1 &   VOT. &    1 &   54 &  1.00 &  0.00 &  0.00 &  1430 &  0.94 &  0.04 &  0.02 &  216 &  0.90 &  0.00 &  0.08 &  0.02 &  0.95 &  0.94 \\
2 &   GRD. &   62 &   85 &  0.47 &  0.36 &  0.16 &  1522 &  0.88 &  0.08 &  0.04 &   93 &  0.99 &  0.00 &  0.00 &  0.01 &  0.88 &  0.86 \\
3 &   CUT. &   18 &  104 &  0.63 &  0.26 &  0.11 &  1367 &  0.98 &  0.00 &  0.02 &  229 &  0.96 &  0.00 &  0.01 &  0.03 &  0.97 &  0.95 \\
4 &  N.BA. &    6 &   93 &  0.98 &  0.00 &  0.02 &  1348 &  0.99 &  0.01 &  0.00 &  259 &  0.93 &  0.00 &  0.01 &  0.05 &  0.98 &  0.98 \\
5 &  H.FN. &   65 &   85 &  0.84 &  0.11 &  0.06 &  1355 &  0.99 &  0.01 &  0.00 &  260 &  0.88 &  0.00 &  0.07 &  0.04 &  0.98 &  0.96 \\
6 &  M.BA. &  149 &   85 &  0.76 &  0.18 &  0.06 &  1364 &  0.98 &  0.01 &  0.01 &  251 &  0.90 &  0.00 &  0.07 &  0.03 &  0.98 &  0.96 \\
7 &   N.N. &  810 &  186 &  0.23 &  0.30 &  0.47 &  1266 &  0.95 &  0.03 &  0.03 &  248 &  0.55 &  0.09 &  0.10 &  0.26 &  0.87 &  0.81 \\
\hline
\end{tabular}
\end{footnotesize}
\end{table}

\begin{table} 
\label{r6}
\begin{footnotesize}
\begin{tabular}{|llr|rrrr|rrrr|rrrrr|rr|}
\hline
T6 &      A &      T &    W &   W:W &   M:W &   N:W &     P &   N:P &   M:P &   W:P &    R &   C:R &   M:R &   W:R &   N:R &    Wk &   Str \\
\hline
0 &   TRV. &    0 &    0 &  0.00 &  0.00 &  0.00 &  1700 &  0.72 &  0.13 &  0.16 &    0 &  0.00 &  0.00 &  0.00 &  0.00 &  0.72 &  0.72 \\
1 &   VOT. &    1 &   97 &  1.00 &  0.00 &  0.00 &  1432 &  0.85 &  0.06 &  0.09 &  171 &  0.77 &  0.01 &  0.22 &  0.01 &  0.87 &  0.85 \\
2 &   GRD. &   64 &  255 &  0.55 &  0.33 &  0.13 &  1402 &  0.84 &  0.07 &  0.09 &   43 &  0.98 &  0.02 &  0.00 &  0.00 &  0.85 &  0.80 \\
3 &   CUT. &   18 &  301 &  0.71 &  0.23 &  0.06 &  1257 &  0.95 &  0.01 &  0.04 &  142 &  0.95 &  0.01 &  0.02 &  0.01 &  0.95 &  0.91 \\
4 &  N.BA. &    6 &  258 &  0.97 &  0.02 &  0.02 &  1224 &  0.98 &  0.02 &  0.00 &  218 &  0.88 &  0.02 &  0.05 &  0.06 &  0.98 &  0.97 \\
5 &  H.FN. &   67 &  255 &  0.82 &  0.10 &  0.07 &  1244 &  0.96 &  0.02 &  0.02 &  201 &  0.82 &  0.01 &  0.13 &  0.03 &  0.95 &  0.92 \\
6 &  M.BA. &  152 &  255 &  0.84 &  0.09 &  0.07 &  1244 &  0.96 &  0.02 &  0.02 &  201 &  0.82 &  0.03 &  0.12 &  0.03 &  0.96 &  0.92 \\
7 &   N.N. &  790 &  138 &  0.30 &  0.12 &  0.57 &  1443 &  0.99 &  0.00 &  0.01 &  119 &  0.44 &  0.00 &  0.16 &  0.40 &  0.91 &  0.89 \\ 
\hline
\end{tabular}
\end{footnotesize}
\end{table}

\begin{table} 
\label{r7}
\begin{footnotesize}
\begin{tabular}{|llr|rrrr|rrrr|rrrrr|rr|}
\hline
T7 &      A &       T &     W &   W:W &   M:W &   N:W &     P &   N:P &   M:P &   W:P &     R &   C:R &   M:R &   W:R &   N:R &    Wk &   Str \\
\hline
0 &   TRV. &     0 &     0 &  0.00 &  0.00 &  0.00 &  8500 &  0.72 &  0.13 &  0.15 &     0 &  0.00 &  0.00 &  0.00 &  0.00 &  0.72 &  0.72 \\
1 &   VOT. &     4 &   819 &  1.00 &  0.00 &  0.00 &  6773 &  0.90 &  0.06 &  0.04 &   908 &  0.78 &  0.01 &  0.17 &  0.04 &  0.92 &  0.90 \\
2 &   GRD. &  1562 &  1275 &  0.57 &  0.30 &  0.13 &  6969 &  0.86 &  0.07 &  0.08 &   256 &  0.95 &  0.02 &  0.00 &  0.04 &  0.86 &  0.82 \\
3 &   CUT. &   574 &  1594 &  0.73 &  0.19 &  0.07 &  6140 &  0.98 &  0.01 &  0.01 &   766 &  0.93 &  0.01 &  0.02 &  0.04 &  0.97 &  0.93 \\
4 &  N.BA. &   148 &  1230 &  0.99 &  0.00 &  0.01 &  6135 &  0.99 &  0.01 &  0.00 &  1135 &  0.87 &  0.02 &  0.02 &  0.08 &  0.98 &  0.97 \\
5 &  H.FN. &  1662 &  1275 &  0.73 &  0.09 &  0.18 &  6202 &  0.95 &  0.02 &  0.04 &  1023 &  0.83 &  0.02 &  0.11 &  0.04 &  0.93 &  0.90 \\
6 &  M.BA. &  3492 &  1275 &  0.70 &  0.08 &  0.21 &  6153 &  0.95 &  0.02 &  0.03 &  1072 &  0.77 &  0.02 &  0.16 &  0.06 &  0.92 &  0.89 \\
\hline
\end{tabular}
\end{footnotesize}
\end{table}
\end{landscape}

\subsection{Discussion}

First, we make three general observations. The first is that most of the algorithms scale by a factor of twenty-five in running time when we go from the small graph to the large graph. This makes sense since we are increasing both the graph size (nodes and edges) by a factor of five and the number of colors by a factor of five. The voting algorithm is an exception. It scales only by a factor of five for the number of nodes and does not scale with number of colors. The differences in the running times of the message passing algorithms can be largely explained by the number of iterations to convergence and not the speed of individual iterations. The second observation is that, as expected, the power law complicates matters: performance is better on the plain circle graph almost accross the board. Finally, we note that in the scaled up graphs with the same error rates several of the algorithms actually \emph{improve} in performance, suggesting that scale is only a challenge of timing and not correctness.

Now we list a series of observations and speculations, broken down by algorithm:

\begin{itemize}
	\item \textbf{Voting: }
Though the voting algorithm technically has close to the weakest performance, beating only gradient descent, there are a few lessons to be learned. It runs extremely quickly. Unlike most other algorithms, it only needs to store a single value at the nodes and not a whole probability vector. With the experimental parameter settings, it was the only algorithm which labeled nodes wild with perfect precision across graph types and sizes. However, this result is mitigated by the fact that its recall was approximately half that of the other algorithms.
	\item \textbf{Gradient: }
The gradient algorithm performs reasonably on the circle model, but once the power law is introduced it detects wilds at a significantly lower rate than the other algorithms with significantly higher error. Its one strength may be relabeling. With the experimental parameter settings, it was very cautious in relabeling, scoring the highest precision among the algorithms. This result is mitigated by the fact that its recall was significantly lower than the other algorithms. The other algorithms could be set to relabel more cautiously with some parameter tweaking.
	\item \textbf{Min Cut: }
The minimum cut is the second-strongest algorithm in terms of weak correctness. It is surprisingly fast. In the larger graph, it takes approximately $8$ seconds per cut. Furthermore, it is not affected as much as the other algorithms when moving from the circle graphs to the power law graphs. It is good at correcting nodes which are mislabeled, consistently relabeling them with higher precision than other algorithms with comparable recall (all but gradient descent). Where it comes up short is in its model of wildness: it consistently labels nodes wild $20\%$ more often than it needs to and suffers a corresponding loss in precision. It should be possible to exploit the fact that many of the nodes being incorrectly classified as wild are in fact mislabeled nodes.
	\item \textbf{Naive Bayes: }
The naive Bayesian algorithm performs the best among all these algorithms in terms of weak and strong correctness. Furthermore, since its fundamental operations is multiplication, as opposed to addition, it converges extremely quickly. Examining the tables, we see its dominance can be attributed to detecting wild nodes with near-perfect precision. We speculate that this happens because the naive Bayesian algorithm has an explicit model for what it means to be wild, while in the other algorithms wildness is an absense of pattern.
	\item \textbf{Harmonic Function: }
Notice the difference between the performance of the random walk in Tables $4$ and $5$. That is, the performance of the random walk drops more sharply when wild nodes are introduced into the graph than when mislabeled nodes are introduced. Intuitively, this makes sense. When a random walk hits a wild node, its performance thereafter gives no useful information about what color its starting node should be. On the other hand, when it hits a mislabeled node its current position is noise but its future movement may still provide useful color information to its start node. In probabilistic terms, having more wilds increases the probability of traversing an uninformative edge and thus ending in a color different from the starting color. Aside from that, the random walk appears to have a consistent issue relabeling nodes which should be left wild and marking wild some nodes which are normal.
	\item \textbf{Multinomial Bayes: }
Given the discussion of similarity between the harmonic function algorithm and the multinomial Bayesian algorithm in convergence conditions, it is not surprising that the multinomial Bayesian algorithm performs similarly to the random walk. We had hoped that by having a model of wildness incorporated throughout and not retrofitted, we would be able to improve upon the random walk performance but it does not appear to be the case. There is one key difference, though. Unlike the harmonic function algorithm, which passes a whole vector as its message, the multinomial Bayesian algorithm only passes a single color (though both still store the full vector in memory at their node). When working on very large graphs in parallel architectures, this fact of smaller message sizes may be important and may affect the choice between otherwise similar algorithms.
	\item \textbf{Neural Network: }
Generally, the performance of the neural network is not impressive, coming in slightly below average in Table $6$. Interestingly, its performance does not drop in Table $6$ from Table $5$. In fact, it improves. To understand why this might be happening, recall that our transductive learning algorithm splits the data into training and testing sets. If a left node is of low degree, then cutting its degree in half may make it behave unreliably in training. For example, consider a left node which is only adjacent to one right node. In the training set, it could be considered a perfect predictor of the color of that right node, even if its hidden right neighbor (on which it will be tested) does not agree in color. A similar argument could be made for a right node of low degree. This might also explain why the neural network relabels many normal nodes or marks them wild. Sparsity causes chaos in the nodes of very low degree.
\end{itemize}

At the risk of over-generalizing, we wish to draw a tentative conclusion from the data as a whole. We suggest that, on average, an algorithm performs better the more tailored its model or objective function is to the problem at hand. For example, the Naive Bayes algorithm seems to succeed by having an explicit model for wildness, while the Minimum Cut algorithm accounts for each type of error explicitly in its objective function. The simplest algorithms, such as the gradient algorithm, which looks for a best-fit, do not have nuances of the problem built in. They are seeking general agreement in a broader sense and thus cannot reach as strong of a result.

\section{Conclusion}
We began our paper by introducing several problems which could be naturally modeled as label correction on large bipartite graphs. We introduced geometric techniques for generating random graph instances with untrustworthy labels and showed that they exhibited desirable properties in terms of locality and degree distributions.

The eight algorithms we tested were: a trivial algorithm, a voting algorithm based on semi-definite programming, a gradient descent algorithm based on a linear best-fit model with quadratic penalty terms, a minimum cut algorithm, a naive Bayesian algorithm, a harmonic function algorithm, a multinomial Bayesian algorithm, and a neural network algorithm. For large graphs with high error rates, the naive Bayesian algorithm emerged victorious when we examined a single success metric, with the minimum cut algorithm coming in close behind and the harmonic function and multinomial Bayesian in third. However, the true story was more nuanced, with certain algorithms performing better in certain metrics and instances.

The niche specialties of these algorithms suggest that the best approach in a particular application may depend heavily on context and that there is potential for a meta-algorithm which incorporates the judgements of several algorithms or which sequentially applies these algorithms, using the outputs of one as priors for another.


\end{document}